\documentclass{amsart}
\usepackage[utf8]{inputenc}
\usepackage{amssymb}
\usepackage[english]{babel}
\usepackage{mathrsfs}
\usepackage{euscript}
\usepackage{graphicx}
\usepackage{orcidlink}

\usepackage{amsaddr}
\usepackage{amssymb}
\usepackage{amsthm}
\usepackage{amsfonts}
\usepackage{amsbsy}
\usepackage{enumitem}
\usepackage{xcolor}
\renewcommand{\le}{\leqslant}
\renewcommand{\ge}{\geqslant}

\newcommand{\ddd}{\,\mathrm{d}}
\newcommand{\lin}{\mathrm{Lin}\,}
\newcommand{\degin}{\mathrm{deg}\,\mathrm{in}\,}
\newcommand{\degout}{\mathrm{deg}\,\mathrm{out}\,}

\renewcommand{\mod}{\,\mathrm{mod}\,}
\newcommand{\rank}{\mathrm{rank}\,}
\newtheorem{theorem}{Theorem}[section]

\newtheorem{lemma}[theorem]{Lemma}

\newtheorem{observation}[theorem]{Observation}
\newtheorem{corollary}[theorem]{Corollary}
\theoremstyle{definition}
\newtheorem{example}[theorem]{Example}

\setlist[enumerate,1]{label=(\roman*)}
\numberwithin{equation}{section}

\begin{document}
	\title[m-isometric composition operators]{$ m $-isometric composition operators on discrete spaces}
	\author{Michał Buchała\,\orcidlink{0000-0001-5272-9600}}
	\address{Doctoral School of Exact and Natural Sciences, Jagiellonian University, Łojasiewicza 11, PL-30348 Kraków, Poland}
	\address{Institute of Mathematics, Jagiellonian University, Łojasiewicza 6, PL-30348 Kraków, Poland}
	\email{michal.buchala@im.uj.edu.pl}


	\begin{abstract}
		In this paper we study composition operators on discrete spaces $ L^{2}(\mu) $. We establish the classification of underlying graphs of such operators. We obtain the characterization of $ m $-isometric composition operators such that the underlying graph has one cycle. The paper contains also the solution of the Cauchy dual subnormality problem in the class of such composition operators. 
	\end{abstract}

   	\maketitle

	\section{Introduction}
The composition operators form the important class of operators. The study of such a class of operators is motivated by the Banach-Stone theorem (see \cite{banachTheorie}), which states that a linear surjective isometry between spaces of continuous functions on compact spaces is in fact a weighted composition operators. The composition operators appear also naturally in ergodic theory (so-called Koopman operator, which is a composition operator induced by measure-preserving transformation). The class of composition operators has been studied extensively for a long time (see e.g. \cite{nordgrenCompositionOperatorsHilbert1978}, \cite{signhCompositionVectorFunctions}). In \cite{budzynskiSubnormalityUnboundedComposition2017} Budzyński et.al. presented an example of unbounded non-hyponormal composition operators having graph with one cycle, which generates the Stieltjes moment sequence. In \cite{pietryckiShimorinTypeAnalyticModel} Pietrzycki obtained an analytic model for a weighted composition operator on discrete space under the assumption that the underlying graph of this operator has finite branching index.\par
$ m $-isometric operators were introduced by Agler in \cite{aglerDisconjugacyTheoremToeplitz1990} and investigated in details in \cite{aglerMisometricTransformationsHilbert1995} (as well as its subsequent parts). From that time on there have been published many papers devoted to study properties of $ m $-isometries (see e.g. \cite{bermudezProductsMisometries2013}, \cite{badeaCauchyDual2isometric2019}, \cite{suciuOperatorsExpansiveMisometric2022}, \cite{jablonskiMisometricOperatorsTheir2020}, \cite{koMisometricToeplitzOperators2018}). In  \cite{abdullahStructureMisometricWeighted2016} and \cite{bermudezWeightedShiftsWhich2010} $ m $-isometric weighted shifts were studied and, eventually, the characterization of weights of $ m $-isometric shift by real monic polynomials of degree at most $ m-1 $ was obtained (see \cite[Theorem 1]{abdullahStructureMisometricWeighted2016}). In \cite{kosmiderMisometricCompositionOperators2021} Jabłoński and Kośmider investigated $ m $-isometricity of composition operator on discrete space such that the graph associated to this operator has one cycle and one branching point.\par
The Cauchy dual $ T' = T(T^{\ast}T)^{-1} $ of a left-invertible operator $ T\in \mathbf{B}(H) $ was introduced by Shimorin  in \cite{shimorinWoldtypeDecompositionsWandering2001} on the occasion of his study of Wold-type decompositions. From \cite[Proposition 6]{athavaleCompletelyHyperexpansiveOperators1996} it follows that the Cauchy dual of a completely hyperexpansive weighted shift is always a subnormal contraction. In \cite{chavanOperatorsCauchyDual2007} there was posed a question whether it is true for every completely hyperexpansive operator. In \cite{anandSolutionCauchyDual2019} Chavan et. al. were considering this problem in the more restricted class of 2-isometries and obtained a negative answer. They also provided several sufficient conditions for 2-isometry to have subnormal Cauchy dual. The Cauchy dual subnormality problem was studied also in \cite{anandCauchyDualSubnormality2020} and \cite{chavanCauchyDualSubnormality2022}.\par
The present paper is devoted to study $ m $-isometric composition operators on discrete spaces in a more general setting than in \cite{kosmiderMisometricCompositionOperators2021}. It is organized as follows. In Section \ref{SecPreliminaries} we introduce the notation and recall basic definitions and some standard results needed in further parts of the paper. The aim of Section \ref{SecClassificationOfCompOperators} is to fully classify self-maps of a countable set. If we associate a certain graph to such a map, it turns out that if this graph is assumed to be connected, then there are only two posibilities for its geometry: either this graph is a rootless directed tree or it consists of rooted directed trees connected by a cycle (see Theorem \ref{ThmClassificationOfGraphs}). In Section \ref{SecMIsometricCompOperators} we investigate the composition operators (viewed as certain weighted shifts) connected to graphs of the second type, that is, the ones having a cycle. We generalize the results presented in \cite{kosmiderMisometricCompositionOperators2021}, where the very particular case of graph with a cycle and the only one branching point was studied. The assumption that the graph has a cycle allows us to obtain certain recursive formulas (see Lemma \ref{LemRecursiveFormulas}), which are key tools in further considerations. The main result of Section \ref{SecMIsometricCompOperators} is Theorem \ref{ThmMIsometricCompositionOperatorsOneCycle}, which gives us the characterization of $ m $-isometricity in terms of the rank of certain matrices. We also prove that the notions of complete hyperexpansivity and 2-isometricity coincide in the class of composition operators having graphs with one cycle. Section \ref{SecCauchyDualOfCompOperator} is devoted to study the Cauchy dual subnormality problem for 2-isometric composition operators on graphs with one cycle. In Theorem \ref{ThmCauchyDualSubnormalityCondition} the surprising characterization of composition operators having subnormal Cauchy dual is presented. It turns out that if the Cauchy dual is subnormal, then for every vertex $ v $ lying on the cycle the measure representing the sequence $ (\lVert T^{n}e_{v}\rVert^{2})_{n=0}^{\infty} $ is two-atomic. Using this result, we provide an example of a composition operator on graph with one cycle, for which the Cauchy dual is not subnormal.
	\section{Preliminaries}
\label{SecPreliminaries}
Denote by $ \mathbb{N} $ and $\mathbb{Z} $ the set of non-negative integers and integers, respectively and by $ \mathbb{R} $ and $ \mathbb{C} $ the field of real and complex numbers, respectively. For $ p \in \mathbb{N} $ we set
\begin{equation*}
	\mathbb{N}_{p} = \{n\in \mathbb{N}\!: n\ge p \}.
\end{equation*}
If $ \mathbb{K}\in \{\mathbb{R},\mathbb{C} \} $, then $ \mathbb{K}_{n}[x] $ stands for the set of all polynomials of degree at most $ n $ with coefficients in $ \mathbb{K} $. From the binomial formula we can derive the following simple fact.
\begin{observation}
	\label{ObsLoweringDegree}
	If $ p\in \mathbb{K}_{n}[x] $ for some $ n\in \mathbb{N}_{1} $ and $ \kappa\in \mathbb{K} $, then
	\begin{equation*}
		p(\,\cdot\,)-p(\,\cdot\,+\kappa) \in \mathbb{K}_{n-1}[x].
	\end{equation*}
\end{observation}
The next lemma will play a crucial role in our considerations.
\begin{lemma}
	\label{LemSeriesOfPolynomials}
	Let $ m\in \mathbb{N}_{1} $. Suppose $ (p_{i})_{i=0}^{\infty}\subset \mathbb{K}_{m}[x] $ and for every $ i\in\mathbb{N} $ write $ p_{i}(x) = \sum_{k=0}^{m}a_{k}^{(i)}x^{k} $. Assume that for every $ k\in\mathbb{N}\cap[0,m] $ the series
	\begin{equation*}
		b_{k}:= \sum_{i=0}^{\infty} a_{k}^{(i)}
	\end{equation*}
	is absolutely convergent. Then the series
	\begin{equation}
		\label{FormSeriesOfPolynomials}
		p(x) := \sum_{i=0}^{\infty} p_{i}(x)
	\end{equation}
	is absolutely convergent for every $ x\in \mathbb{K} $ and $ p\in \mathbb{K}_{m}[x] $. Moreover,
	\begin{equation*}
		p(x) = \sum_{i=0}^{m} b_{k}x^{k}, \qquad x\in \mathbb{R}.
	\end{equation*}
\end{lemma}
\begin{proof}
	If $ x\in \mathbb{K} $ and $ n\in \mathbb{N} $, then
	\begin{align*}
		\sum_{i=0}^{n} \lvert p_{i}(x)\rvert \le \sum_{k=0}^{m} \lvert x\rvert ^{k}\sum_{i=0}^{n} \lvert a_{k}^{(i)}\rvert.
	\end{align*}
	Hence, the series \eqref{FormSeriesOfPolynomials} is absolutely convergent. Set $ q(x) = \sum_{k=0}^{m} b_{k}x^{k} $, $ x\in \mathbb{K} $. For every $ x\in \mathbb{K} $ we have
	\begin{equation*}
		\sum_{i=0}^{n} p_{i}(x) = \sum_{k=0}^{m} x^{k}\sum_{i=0}^{n} a_{k}^{(i)} \stackrel{n\to\infty}{\longrightarrow} \sum_{k=0}^{m}b_{k}x^{k} = q(x).
	\end{equation*}
	Thus, $ p(x) = q(x) $ for $ x\in \mathbb{K} $. Consequently, $ p\in \mathbb{K}_{m}[x] $.
\end{proof}
For a topological space $ X $ we denote by $ \mathcal{B}(X) $ the $ \sigma $-algebra of Borel subsets of $ X $; if $ x\in X $, then by $ \delta_{x} $ we denote the Dirac delta measure. For $ A\subset X $, $ \chi_{A} $ stands for characteristic function of $ A $; we abbreviate $ \chi_{x} = \chi_{\left\{ x \right\}} $ for $ x\in X $. If $ H $ is a complex Hilbert space, then $ \mathbf{B}(H) $ stands for the $ C^{\ast} $-algebra of all linear and bounded operators on $ H $.\par
By a discrete measure space we mean a pair $ (X,\mu) $, where $ X $ is an at most countable set and $ \mu\!: 2^{X}\to [0,\infty] $ is a positive measure on $ X $ such that $ \mu(\{x\}) \in (0,\infty) $ for every $ x\in X $; for simplicity we write $ \mu(x) := \mu(\{x\}) $ for $ x\in X $. If $ (X,\mu) $ is a discrete measure space and $ T\!: X\to X $ is a transformation on $ X $, then $ \mu\circ T^{-1} $ is absolutely continuous with respect to $ \mu $, where $ \mu\circ T^{-1}\!: 2^{X}\to [0,\infty] $ is a positive measure on $ X $ defined by
\begin{equation*}
	\mu\circ T^{-1}(A) = \mu(T^{-1}(A)), \qquad A\subset X.
\end{equation*}
Denote by $ h_{T} $ the Radon-Nikodym derivative of $ \mu\circ T^{-1} $ with respect to $ \mu $; we can easily compute that
\begin{equation}
	\label{FormRadonNikodymDerivativeDiscrete}
	h_{T}(x) = \frac{\mu(T^{-1}(\{x\}))}{\mu(x)}, \qquad x\in X.
\end{equation}
Assuming additionally that $ h_{T}\in L^{\infty}(\mu) $, we define the composition operator $ C_{T}\in \mathbf{B}(L^{2}(\mu)) $ as follows:
\begin{equation*}
	C_{T}f = f\circ T, \qquad f\in L^{2}(\mu).
\end{equation*}
For more details about general theory of composition operators we refer to \cite{nordgrenCompositionOperatorsHilbert1978}.\par
We say that $ S\in \mathbf{B}(H) $ is normal if $ S^{\ast}S = SS^{\ast} $; $ S $ is called subnormal if there exists a Hilbert space $ K $ containing (an isometric embedding of) $ H $ and a normal operator $ N\in \mathbf{B}(K) $ such that $ N|_{H} = S $. Recall that $ S\in \mathbf{B}(H) $ is completely hyperexpansive if
\begin{equation*}
    \sum_{k=0}^{n}(-1)^{k}\binom{n}{k} S^{\ast k}S^{k} \le 0, \qquad n\in\mathbb{N}_{1}.
\end{equation*}
We say that $ S\in \mathbf{B}(H) $ is $ m $-isometric ($ m\in\mathbb{N}_{1} $), if
\begin{equation*}
	\sum_{k=0}^{m} (-1)^{m-k}\binom{m}{k} S^{\ast k}S^{k} = 0.
\end{equation*}
For an at most countable set $ V $ we define
\begin{equation*}
	\ell^{2}(V) = \left \{f\!: V\to V | \sum_{v\in V}\lvert f(v)\rvert^{2}<\infty \right \};
\end{equation*}
this is a Hilbert space with the norm given by:
\begin{equation*}
	\lVert f\rVert = \sqrt{\sum_{v\in V}\lvert f(v)\rvert^{2}}, \qquad f\in \ell^{2}(V).
\end{equation*}
The vectors $ e_{v} $ ($ v\in V $) defined as follows:
\begin{equation*}
	e_{v}(u) = \begin{cases}
		1, & \text{ if } u = v\\
		0, & \text{ otherwise}
	\end{cases},
\end{equation*}
form an orthonormal basis of $ \ell^{2}(V) $.\par
In the subsequent part of this section we introduce some basic notions of graph theory (for more details see \cite{oreTheoryGraphs1962} and \cite{ahoDataStructuresAlgorithms1983}). By a directed graph we mean a pair $ G = (V,E) $, where $ V $ is a non-empty at most countable set and $ E\subset V\times V $; elements of $ V $ are called vertices and elements of $ E $ are called edges. Every pair $ G_{1} = (V_{1},E_{1}) $ such that $ V_{1}\subset V $ and $ E_{1}\subset (V_{1}\times V_{1})\cap E $ is called a subgraph of $ G $. For a vertex $ v\in V $ we define its in-degree, out-degree and degree as follows:
\begin{align*}
	\degin v &= \#\{w\in V\!:(w,v)\in E \}\\
	\degout v &= \#\{w\in V\!: (v,w)\in E \}\\
	\deg v &= \degin v + \degout v.
\end{align*}
If $ v,w\in V $, then a path from $ v $ to $ w $ in $ G $ is a sequence $ (v_{0},\ldots,v_{k}) \subset V $ such that $ v_{0} = v $, $ v_{k} = w $ and $ (v_{i},v_{i+1})\in E $ for every $ i\in\mathbb{N}\cap [0,k-1] $. A path $ (v_{0},\ldots,v_{k}) $ is called a cycle if $ v_{0} = v_{k} $; the cycle is said to be simple if the vertices $ v_{j} $ ($ j\in \mathbb{N}\cap[0,k-1] $) are distinct. \par
Graph $ G $ is said to be connected if for every two distinct vertices $ v,w\in V $ there exists a sequence $ (v_{0},\ldots,v_{k}) \subset V $ such that $ v_{0} = v $, $ v_{k} = w $ and $ \{v_{i},v_{i+1}\}\in \tilde{E} $ for $ i\in \mathbb{N}\cap[0,k-1] $, where 
\begin{equation*}
	\tilde{E} = \{\{v,w\}\subset V\!: (v,w)\in E \text{ or } (w,v)\in E \}.
\end{equation*} If $ G = (V,E) $ is a directed graph, then its subgraph $ G_{1} = (V_{1}, E_{1}) $ is called a connected component if $ G_{1} $ is connected and for each connected subgraph $ G_{1}'=(V_{1}',E_{1}') $ of $ G $ satisfying $ V_{1}\subset V_{1}' $ and $ E_{1}\subset E_{1}' $ we have $ G_{1} = G_{1}' $. In other words, connected components of $ G $ are maximal connected subgraphs of $ G $. \par
A connected graph $ G = (V,E) $ is called a directed tree if $ G $ has no cycles and for every $ v\in V $ we have $ \degin v \le 1 $. If $ G $ is a directed tree, then there is at most one vertex $ \omega\in V $ such that $ \degin \omega = 0 $; if it exists, we call this vertex the root of $ G $ (see \cite[Proposition 2.1.1]{jablonskiWeightedShiftsOnDirectedTrees}). The reader can verify that for rooted directed trees the following induction principle holds.
\begin{lemma}
	\label{LemInductionOnTree}
	Let $ G = (V,E) $ be a directed tree with root $ \omega $. Let $ V'\subset V $ satisfy the following conditions:
	\begin{enumerate}
		\item $ \omega\in V' $,
		\item for every $ v\in V' $ we have $ \{u\in V\!: (v,u)\in E\} \subset V' $.
	\end{enumerate}
	Then $ V' = V $.
\end{lemma}
We say that a graph $ G $ is strongly connected if for any two distinct vertices $ v,w\in V $ there is a path from $ v $ to $ w $ \textbf{and} a path from $ w $ to $ v $ (in other words, $ v $ and $ w $ lie on a common cycle). For a directed graph $ G = (V,E) $ we define a relation $ \mathcal{SC}_{G}\subset V\times V $ in the following way: $ (v,w)\in \mathcal{SC}_{G} $ if and only if either $ v= w $ or there exists a path from $ v $ to $ w $ in $ G $ and a path from $ w $ to $ v $ in $ G $. It is a matter of routine to verify that $ \mathcal{SC}_{G} $ is an equivalence relation on $ V $. Each equivalence class $ C\subset V $ of $ \mathcal{SC}_{G} $ induces a subgraph $ G_{C} = (C,E_{C}) $, where $ E_{C} = E\cap (C\times C) $, which is called a strongly connected component of $ G $.
	\section{Classification of graphs of self-maps of countable sets}
\label{SecClassificationOfCompOperators}

In \cite[Lemma 4.3.1]{jablonskiNonhyponormalOperatorGenerating2012} there was proved that weighted shifts on rootless directed trees with positive weights are unitarily equivalent to composition operators on certain discrete spaces (see also \cite[Lemma 3.1.4]{budzynskiSubnormalityUnboundedComposition2017}). In \cite[Theorem 3.2.1]{budzynskiSubnormalityUnboundedComposition2017} the authors classified composition operators on discrete spaces, for which the associated graph is connected and has at most one vertex of out-degree greater than 1. In \cite[Proposition 2.4]{stochel2WeightedQuasishifts1990} composition operators on graphs with the property that every vertex has out-degree 1 are characterized. In this section, using graph-theoretical approach, we give the full classification of graphs of self-maps of at most countable set. This classification was mentioned in \cite[p.898]{pietryckiShimorinTypeAnalyticModel}; however, the proof was not provided there.\par
Suppose $ X $ is an at most countable set and $ T\!: X\to X $ is a transformation on X; for $ x\in X $ we abbreviate: $ Tx = T(x) $. Consider the graph $ G_{T} = (V_{T},E_{T}) $, where $ V_{T} = X $ and $ E_{T} = T^{-1} = \left\{ (Tx,x)\!: x\in X \right\} $. In the next few results we investigate the geometry of graph $ G_{T} $. 
\begin{observation}
	\label{ObsIndegreeOfGraphUniquePathWithDistinctVertices}
	Suppose $ X $ is an at most countable set and $ T\!: X\to X $. Then, for every $ v\in V_{T} $ we have $ \degin v = 1 $. Moreover, if for $ v,w\in V_{T} $, $ v\not= w $,  there exists a path from $ v $ to $ w $, then there is the unique path from $ v $ to $ w $ consisting of distinct vertices.
\end{observation}
\begin{proof}
	The first part can be easily derived from the definition of $ G_T $. Let us prove the ''moreover'' part. Obviously, we can always find a path from $ v $ to $ w $ consisting of distinct vertices, so we only have to prove uniqueness. Let $ v = v_0,v_1,\ldots,v_{n},v_{n+1} = w $ and $ v = u_{0},u_1,u_2,\ldots,u_{m}, u_{m+1} = w $ be paths connecting $ v $ and $ w $ such that $ v_{j}\not= v_{k} $ for every $ j,k\in\mathbb{N}\cap [0,n+1] $, $ j\not= k $ and $ u_{j}\not= u_{k} $, $ j,k\in\mathbb{N}\cap [0,m+1] $ for every $ j\not= k $; without loss of generality we can assume $ n\le m $. Since $ \degin w = 1 $, we have that $ v_{n} = u_{m} $. Proceeding by induction, we can show that $ v_{n-k} = u_{m-k} $ for all $ k\in \mathbb{N}\cap[0,n] $. In particular, we have $ v_{0} = u_{m-n} $. This, by the choice of $ u_{0},\ldots,u_{m} $, implies that $ m = n $ and $ v_{j} = u_{j} $ for all $ j\in \mathbb{N}\cap[0,n] $.
\end{proof}
The proceeding results reveal the structure of strongly connected components of $ G_T $. Lemma \ref{LemUniqueSCC} shows that every connected component of $ G_T $ may contain at most one non-trivial strongly connected component.
\begin{lemma}
	\label{LemUniqueSCC}
	Suppose $ X $ is an at most countable set and $ T\!: X\to X $ is such that $ G_T $ is connected. Then $ G_{T} $ contains at most one strongly connected component $ G_{0} $ with non-empty set of edges (that is, $ G_{0} \not= (\{v\},\varnothing) $, where $ v\in V_{T} $).
\end{lemma}
\begin{proof}
	Suppose to the contrary that $ G_0 = (V_0,E_0) $ and $ G_1 = (V_1,E_1) $ are two different strongly connected components of $ G_T $ with non-empty sets of edges. Obviously, $ V_0\cap V_1 = \varnothing $. Let $ v_0\in V_0 $ and $ v_1\in V_1 $. We will show that there exists a path from $ v_{0} $ to $ v_{1} $. Since $ G_T $ is connected, there exists a sequence $ (u_{0}, \ldots, u_{n+1}) \subset V_{T} $ ($ n\in \mathbb{N} $) such that $ u_{0} = v_{0} $,  $ u_{n+1} = v_{1} $ and $ \{u_{k},u_{k+1}\}\in \tilde{E}_{T} $ for every $ k\in \mathbb{N}\cap[0,n] $. If $ (u_{1},v_{0}) \in E_{T} $, then from the fact that $ G_{0} $ is strongly connected, $ E_{0}\not=\varnothing $ and $ \degin v_{0} = 1 $ we obtain that $ u_{1}\in V_{0} $. This implies that there exists a path from $ v_{0} $ to $ u_{1} $ in $ G_{T} $. Hence, possibly after extending a sequence $ (u_{0},\ldots, u_{n+1}) $, without loss of generality we can assume that $ (u_{0},u_{1})\in E_{T} $. If $ (u_{k},u_{k+1})\in E_{T} $ for every $ k\in \mathbb{N}\cap[0,n] $, then the sequence $ (u_{0},\ldots,u_{n+1}) $ is the path from $ v_{0} $ to $ v_{1} $. Suppose that it is not the case; in particular, $ n\ge 1 $. Set
	\begin{equation*}
		k_{0} = \min\{k\in \mathbb{N}\cap[1,n]\!: (u_{k+1},u_{k})\in E_{T} \}.
	\end{equation*}
	Since $ (u_{k_{0}+1,}u_{k_{0}})\in E_{T} $, $ (u_{k_{0}-1},u_{k_{0}})\in E_{T} $ and $ \degin u_{k_{0}} = 1 $, it follows that $ u_{k_{0}+1} = u_{k_{0}-1} $. If $ k_{0} = n $, then we have $ u_{n-1} = u_{n+1} = v_{1} $. Thus, $ (u_{0},\ldots,u_{n-1}) $ is the path from $ v_{0} $ to $ v_{1} $. If $ k_{0} < n $, then the sequence
	\begin{equation*}
		(u_{0}',\ldots,u_{n}') = (u_{0},\ldots,u_{k_{0}-1},u_{k_{0}+2},\ldots,u_{n+1})
	\end{equation*}
	have the property that $ \{u_{k}',u_{k+1}'\}\in \tilde{E}_{T} $ for every $ k\in \mathbb{N}\cap[0,n-1] $, so the above procedure can be applied to this sequence. Proceeding this way, we eventually obtain the path from $ v_{0} $ to $ v_{1} $. By symmetry, there exists also a path from $ v_{1} $ to $ v_{0} $. Since $ G_{0} $ and $ G_{1} $ are strongly connected, we have $ v_{1}\in V_{0} $ and $ v_{0}\in V_{1} $, which is a contradiction.
\end{proof}
The next result shows that non-trivial strongly connected components of $ G_T $ are of very special type.
\begin{lemma}
	\label{LemSCCIsSimpleCycle}
	Suppose $ X $ is an at most countable set and $ T\!: X\to X $ is such that $ G_T $ is connected and has a unique strongly connected component $ G_0 = (V_0,E_0) $ with non-empty set of edges. Then $ G_0 $ is a simple cycle.
\end{lemma}
\begin{proof}
	If $ \#V_0 = 1 $, then the conclusion obviously holds. Assume $ \#V_0 \ge 2 $. Let $ v,w\in V_{0} $, $ v\not= w $. By the strong connectedness of $ G_0 $, $ v $ and $ w $ lie on the common cycle; we will prove that this cycle is actually the whole graph $ G_{0} $. By Observation \ref{ObsIndegreeOfGraphUniquePathWithDistinctVertices} there exist paths from $ v $ to $ w $ and $ w $ to $ v $ consisting of distinct vertices; call by $ v = v_0,v_1,\ldots,v_n,v_{n+1} = w $ the vertices of such a path from $ v $ to $ w $ and by $ w = w_{0},w_1,\ldots,w_{m},w_{m+1} = v $ the vertices of such a path from $ w $ to $ v $. Set $ V' = \{v,v_1,\ldots,v_{n},w,w_{1},\ldots,w_{m} \} \subset V_0 $. It is enough to show that $ V' = V_0 $. Suppose to the contrary that it is not the case and let $ u\in V_{0}\setminus V' $. Let $ u = u_0,u_1,\ldots,u_{\ell},u_{\ell+1} = w $ be the path connecting $ u $ and $ w $ and consisting of distinct vertices. If $ \ell \le n $, then the reasoning as in the proof of Observation \ref{ObsIndegreeOfGraphUniquePathWithDistinctVertices} gives us that $ u_{i} = v_{i+n-\ell} $ for $ i\in \mathbb{N}\cap [0,\ell] $, which implies that $ u = u_{0} = v_{n-\ell} \in V_{0} $. If $ \ell > n $, then the same reasoning gives us that $ v = u_{\ell-n} $. Since it necessarily holds that $ \ell-n \le m $, we have $ u = u_{0} = w_{m+1+n-\ell}\in V_{0} $. In both cases we get a contradiction.
\end{proof}
The following lemma says that after removing non-trivial strongly connected component from $ G_T $ we actually obtain a family of directed trees.
\begin{lemma}
	\label{LemAfterRemovingSCC}
	Suppose $ X $ is an at most countable set and $ T\!: X\to X $ is such that $ G_T $ is connected and has a unique strongly connected component $ G_0 = (V_0,E_0) $ with non-empty set of edges. Set $ G' = (V_{T}\setminus V_{0},E') $, where $ E' = \{(u,v)\in E_{T}\!: u,v\in V_{T}\setminus V_0 \} $. If $ G_{0}' = (V_{0}', E_{0}') $ is a connected component of $ G' $, then $ G_{0}' $ is a rooted directed tree.
\end{lemma}
\begin{proof}
	From Lemma \ref{LemUniqueSCC} it follows that $ G_0' $ has no cycles. This, together with Observation \ref{ObsIndegreeOfGraphUniquePathWithDistinctVertices}, implies that $ G_{0}' $ is a directed tree. It remains to prove that $ G_{0}' $ has a root. Let $ v\in V_{0} $ and $ w \in V_{0}' $. Proceeding as in the proof of Lemma \ref{LemUniqueSCC} we obtain that there exists a path from $ v $ to $ w $ in $ G_{T} $. Let $ v = v_{0},v_{1},\ldots, v_{m}, v_{m+1} = w $ be the path from $ v $ to $ w $ consisting of distinct vertices. Set $ M = \min\{k\in \mathbb{N}\cap[0,m+1]\!: v_{k}\notin V_{0} \} $. From the fact that $ \degin v_{M} = 1 $ and that $ (v_{M-1},v_{M}) \in E_{T} $ is the only edge coming to $ v_{M} $ it follows that $ v_{M} $ has in-degree 0 in $ G_{0}' $. This implies that $ v_{M} $ is the root of $ G_{0}' $.
\end{proof}

Now we are in the position to state the main results of this section, which classifies graphs of self-maps $ T\!: X\to X $.
\begin{theorem}
	\label{ThmClassificationOfGraphs}
	Suppose $ X $ is an at most countable set and $ T\!: X\to X $ is such that $ G_{T} $ is connected. Then exactly one of the following holds:
	\begin{enumerate}
		\item $ G_{T} $ is a rootless directed tree,
		\item $ G_{T} $ satisfies the conditions below:
		\begin{enumerate}
			\item there exist $ \kappa\in \mathbb{N}_{1} $ and distinct vertices $ v_{0}\ldots,v_{\kappa-1}\in V_{T} $ such that $ (v_{j},v_{j+1})\in E_{T} $ for $ j\in\mathbb{N}\cap [0,\kappa-2] $ and $ (v_{\kappa-1},v_{0})\in E $
			\item there exists a partition $ \mathcal{R} $ of $ V_{T}\setminus\{v_{0},\ldots,v_{\kappa-1}\} $ such that for every $ R\in \mathcal{R} $ the graph $ (R,E_{R}) $, where $ E_{R} = \{(u,w)\in E_{T}\!: u,w\in R \} $, is a rooted directed tree with the root $ \omega_{R} $ having the property that there exists exactly one $ i_{R}\in\mathbb{N}\cap[0,\kappa-1] $ satisfying $ (v_{i_{R}},\omega_{R})\in E_{T} $.
		\end{enumerate}
	\end{enumerate}
\end{theorem}
\begin{proof}
	First, suppose $ G_{T} $ has no non-trivial strongly connected components. This implies that $ G_{T} $ has no cycles. By Observation \ref{ObsIndegreeOfGraphUniquePathWithDistinctVertices}, $ G_{T} $ is a rootless directed tree, so in this case (i) holds. Assuming now that $ G_{T} $ has a non-trivial strongly connected component, by calling Lemmata \ref{LemSCCIsSimpleCycle} and \ref{LemAfterRemovingSCC}, we derive that $ G_{T} $ satisfies the conditions in (ii).
\end{proof}
In Figure \ref{ImGraphWithCycle} we present the example of graph $ G_{T} $ satisfying Theorem \ref{ThmClassificationOfGraphs}.(ii); for the convenience, we denote $ \{R_{i,k}\}_{k=1}^{N_{i}} = \{R\in \mathcal{R}\!: i_{R} = i \} $ for $ i\in \mathbb{N}\cap [0,\kappa-1] $, where $ N_{i} $ is the cardinality of the set on the right hand side.
\begin{figure}
	\centering
	\includegraphics[scale=0.5]{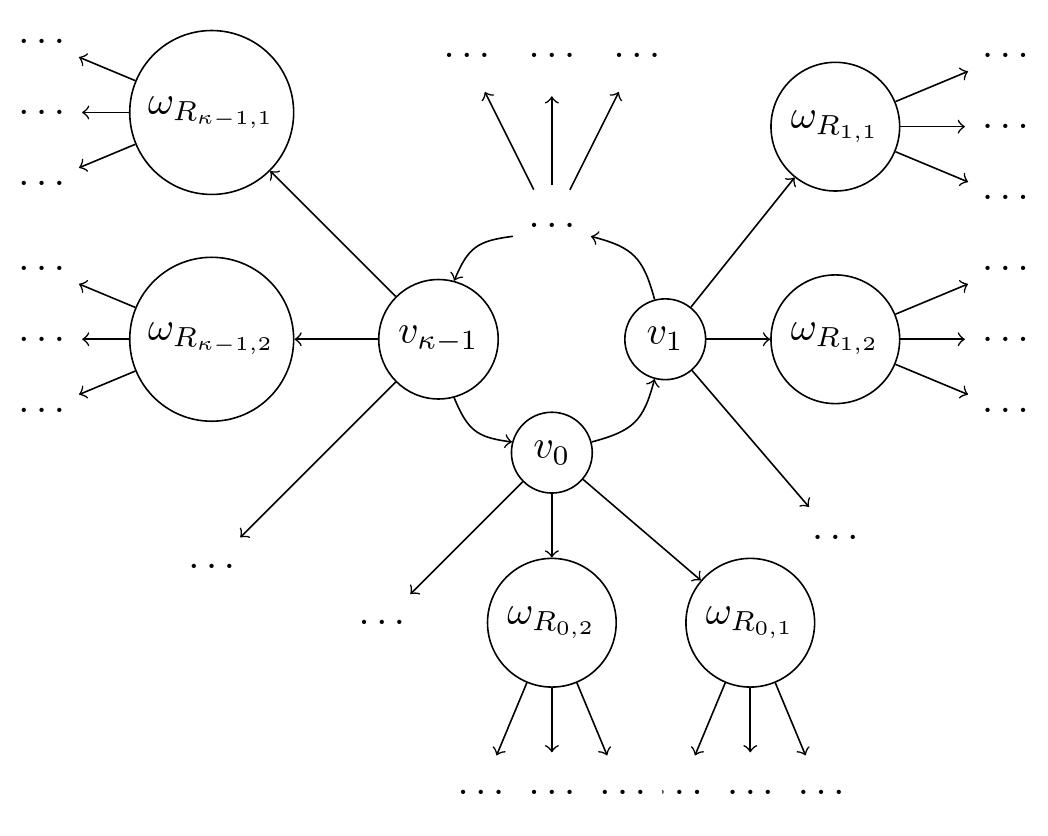}
	\caption{Graph $ G_{T} $ satisfying Theorem \ref{ThmClassificationOfGraphs}.(ii)}
	\label{ImGraphWithCycle}
\end{figure}

	\section{Composition operators as weighted shifts on directed graphs}
In this section we present how to transform a composition operator on discrete space into a weighted shift on a directed graph; we prove several routine results about this weighted shift.
Suppose that $ (X,\mu) $ is a discrete measure space and $ T\!: X\to X $ is a transformation on $ X $ such that $ h_{T}\in L^{\infty}(\mu) $. Define the family of weights $ \pmb{\lambda}_{T} = (\lambda_{v})_{v\in V_{T}} $ by
\begin{equation}
	\label{FormDefinitionOfWeights}
	\lambda_{v} = \frac{\sqrt{\mu(v)}}{\sqrt{\mu(Tv)}}, \qquad v\in V_{T},
\end{equation}
From \eqref{FormRadonNikodymDerivativeDiscrete} it can be easily derived that
\begin{equation*}
	\sup_{u\in V} \sum_{\substack{v\in V_{T}\\ Tv = u}} \lambda_{v}^{2} < \infty \iff h_{T}\in L^{\infty}(\mu).
\end{equation*}
Since we assume that $ h_{T}\in L^{\infty}(\mu) $, the above allows us to define the operator $ S_{\pmb{\lambda}_{T}}\in \mathbf{B}(\ell^{2}(V_{T})) $ as follows:
\begin{equation*}
	S_{\pmb{\lambda}_{T}} e_{u} = \sum_{\substack{v\in V_{T}\\ Tv = u}} \lambda_{v}e_{v}, \qquad u\in V_{T}.
\end{equation*}
It is a matter of routine to prove the following lemma.
\begin{lemma}
	\label{LemUnitaryEquivalence}
	Let $ (X,\mu) $ be a discrete measure space. Suppose $ T\!: X\to X $ is such that $ h_{T}\in L^{\infty}(\mu) $. Then the operator $ U\!: \ell^{2}(V_{T})\to L^{2}(\mu) $ given by
	\begin{equation*}
		Ue_{v} = \frac{1}{\sqrt{\mu(v)}}\chi_{v}, \qquad v\in V_{T} = X,
	\end{equation*}
	makes $ C_{T} $ and $ S_{\pmb{\lambda}_{T}} $ unitarily equivalent.
\end{lemma}
\begin{proof}
	First, we show that $ U $ is unitary. For $ v,w\in V_{T} $ we have
	\begin{equation*}
		\langle Ue_{v}, Ue_{w}\rangle = \left\langle \frac{1}{\sqrt{\mu(v)}}\chi_{v}, \frac{1}{\sqrt{\mu(w)}}\chi_{w}\right\rangle =\frac{1}{\sqrt{\mu(v)}\cdot \sqrt{\mu(w)}}\int_{X}\chi_{v}\cdot \chi_{w}\ddd \mu.
	\end{equation*}
	If $ v\not=w $, then $ \int_{X}\chi_{v}\cdot \chi_{w}\ddd\mu = 0 $; if $ v = w $, then
	\begin{equation*}
		\frac{1}{\sqrt{\mu(v)}\cdot \sqrt{\mu(w)}}\int_{X}\chi_{v}\cdot \chi_{w}\ddd \mu = \frac{1}{\mu(v)}\cdot \mu(v) = 1.
	\end{equation*}
	Hence, 
	\begin{equation*}
		\langle Ue_{v}, Ue_{w}\rangle = \begin{cases}
			0, & \text{ if } v\not=w\\
			1, & \text{ if } v = w
		\end{cases},
	\end{equation*}
	which implies that $ \langle Ue_{v}, Ue_{w}\rangle = \langle e_{v},e_{w}\rangle $. Since $ \{e_{v}\}_{v\in V_{T}} $ is the orthonormal basis of $ \ell^{2}(V_{T}) $, it follows that $ \langle Uf,Ug\rangle = \langle f,g\rangle $ for $ f,g\in \ell^{2}(V_{T}) $, that is, $ U $ is isometry. In turn, $ \overline{\lin\{\chi_{v}\!: v\in V_{T}\}} = L^{2}(\mu) $, so $ U(\ell^{2}(V_{T})) = L^{2}(\mu) $. Therefore, $ U $ is unitary. Next, if $ v\in V_{T} $, then
	\begin{equation*}
		C_{T}Ue_{v} = \frac{1}{\sqrt{\mu(v)}}\cdot \chi_{v}\circ T = \frac{1}{\sqrt{\mu(v)}}\chi_{T^{-1}(\{v\})}
	\end{equation*}
	and
	\begin{align*}
		US_{\pmb{\lambda}_{T}}e_{v} &= U \sum_{\substack{u\in V_{T}\\ Tu=v}} \lambda_{u}e_{u}\\
		&=\sum_{\substack{u\in V_{T}\\ Tu=v}} \frac{\sqrt{\mu(u)}}{\sqrt{\mu(v)}}\cdot \frac{1}{\sqrt{\mu(u)}}\chi_{u}\\
		&= \frac{1}{\sqrt{\mu(v)}}\sum_{\substack{u\in V_{T}\\ Tu=v}} \chi_{u} = \frac{1}{\sqrt{\mu(v)}}\chi_{T^{-1}(\{v\})}.
	\end{align*}
	Hence, $ C_{T}U = US_{\pmb{\lambda}_{T}} $, what completes the proof.
\end{proof}
The lemma below states that connected components of $ G_{T} $ induce reducing subspaces of $ S_{\pmb{\lambda}_{T}} $, so we can restrict our investigation to the case when the graph $ G_{T} $ is connected.
\begin{lemma}
	\label{LemCCAreInvariantSubspaces}
	Suppose $ (X,\mu) $ is a discrete measure space and $ T\!: X\to X $ is a transformation on $ X $ such that $ h_{T}\in L^{\infty}(\mu) $. Let $ G_{0} = (V_{0}, E_{0}) $ be a connected component of $ G_{T} $. Then the space $ M = \overline{\lin\{e_{v}\!: v\in V_{0} \}} $ is a reducing subspace for $ S_{\pmb{\lambda}_{T}} $. 
\end{lemma}
\begin{proof}
	Let $ v\in V_{0} $. Since $ G_{0} $ is a connected component of $ G_{T} $ we have that for every $ u\in V_{T} $ satisfying $ (v,u)\in E_{T} $, $ u\in V_{0} $ and $ (v,u)\in E_0 $. Hence, $ S_{\pmb{\lambda}_{T}}e_{v} \in M $, which implies that $ S_{\pmb{\lambda}_{T}}M \subset M $. Next, we will prove that $ M^{\perp} $ is also invariant for $ S_{\pmb{\lambda}_{T}} $. It can be easily seen that $ (u,v),(v,u)\notin E $ for every $ v\in V_{0} $ and every $ u\in V_{T}\setminus V_{0} $. From this it follows that $ S_{\pmb{\lambda}_{T}}e_{u} \in M^{\perp} $ for every $ u\in V_{T}\setminus {V_{0}} $. This implies that $ S_{\pmb{\lambda}_{T}} M^{\perp} \subset M^{\perp} $, what completes the proof.
\end{proof}
It turns out that that the operator $ S_{\pmb{\lambda}} $ preserves the orthogonality of the standard orthonormal basis of $ \ell^{2}(V_{T}) $.
\begin{observation}
	\label{ObsPreservingOrthogonality}
	Suppose $ (X,\mu) $ is a discrete measure space and $ T\!: X\to X $ is a transformation on $ X $ such that $ h_{T}\in L^{\infty}(\mu) $ and that $ G_{T} $ is connected. Then $ S_{\pmb{\lambda}_{T}} e_{u}\perp S_{\pmb{\lambda}_{T}} e_{v} $ for every vertices $ u,v\in V_{T} $, $ v\not= u $.
\end{observation}
\begin{proof}
	Let $ v,u \in V_{T} $ be distinct vertices. By Observation \ref{ObsIndegreeOfGraphUniquePathWithDistinctVertices}, we have
	\begin{equation*}
		\{w\in V_{T}\!: (v,w)\in E_{T} \} \cap \{w\in V_{T}\!: (u,w)\in E_{T} \} = \varnothing.
	\end{equation*}
	Since $ e_{w_{1}}\perp e_{w_{2}} $ for every two distinct vertices $ w_{1},w_{2}\in V_{T} $, we deduce from the above that
	\begin{equation*}
		\langle S_{\pmb{\lambda}_{T}}e_{u},S_{\pmb{\lambda}_{T}}e_{v}\rangle = \sum_{\substack{w_{1}\in V_{T}\\(u,w_{1})\in E_{T}}}\sum_{\substack{w_{2}\in V_{T}\\(v,w_{2})\in E_{T}}} \lambda_{w_{1}}\overline{\lambda_{w_{2}}}\langle e_{w_{1}},e_{w_{2}}\rangle = 0.
	\end{equation*}
	Hence, $ S_{\pmb{\lambda}_{T}} e_{u}\perp S_{\pmb{\lambda}_{T}} e_{v} $.
\end{proof}
Set\footnote{We stick to the convention that $ \prod \varnothing = 1 $.}
\begin{equation*}
	\Lambda_{m,i} := \prod_{j=m+1}^{m+i} \lambda_{v_{j\mod \kappa}}, \qquad m\in\mathbb{N}\cap[0,\kappa-1], i\in\mathbb{N}.
\end{equation*}
For the further use we emphasize several formulas.
\begin{lemma}
	\label{LemRecursiveFormulas}
	Suppose $ (X,\mu) $ is a discrete space and $ T\!: X\to X $ is such that $ h_{T}\in L^{\infty}(\mu) $ and that $ G_{T} $ is connected and satisfies Theorem \ref{ThmClassificationOfGraphs}.(ii). 
	\begin{enumerate}
		\item For every $ m\in \mathbb{N}\cap[0,\kappa-1] $, $ \Lambda_{m,\kappa} = 1 $.
		\item For every $ k\in\mathbb{N} $ and every $ m\in\mathbb{N}\cap[0,\kappa-1] $ we have
		\begin{align}
			\label{FormNormOfSquareSmallPowers}
			S_{\pmb{\lambda}_{T}}^{k}e_{v_{m}} &= \Lambda_{m,k} e_{v_{(m+k)\mod \kappa}} + \sum_{i = 0}^{k-1} \Lambda_{m,i} \!\!\!\!\!\sum_{\substack{R\in\mathcal{R}\\i_{R} = i+m\mod\kappa}}\!\!\!\!\! \lambda_{\omega_{R}}S_{\pmb{\lambda}_{T}}^{k-i-1}e_{\omega_{R}};
		\end{align}
		\item For every $ n\in\mathbb{N} $ and every $ m\in\mathbb{N}\cap[0,\kappa-1] $ the following recursive formula is satisfied:
		\begin{align}
			\label{FormRecursiveNormOfSquare}
			S_{\pmb{\lambda}_{T}}^{n+\kappa}e_{v_{m}} &= S_{\pmb{\lambda}_{T}}^{n}e_{v_{m}} + \sum_{i = 0}^{\kappa-1}\Lambda_{m,i}\!\!\!\!\!\sum_{\substack{R\in\mathcal{R}\\i_{R} = i+m\mod\kappa}}\!\!\!\!\! \lambda_{\omega_{R}}S_{\pmb{\lambda}_{T}}^{n+\kappa-i-1}e_{\omega_{R}}.
		\end{align}
	\end{enumerate}
\end{lemma}
\begin{proof}
	(i). It follows from the definition of weights $ \lambda_{i} $, $ i\in \mathbb{N}\cap[0,\kappa-1] $.\\
	(ii). We prove it by induction. For $ k=0,1 $ the equality \eqref{FormNormOfSquareSmallPowers} holds. If \eqref{FormNormOfSquareSmallPowers} holds for $ k\in \mathbb{N} $, then
	\begin{align*}
		S^{k+1}_{\pmb{\lambda}_{T}}e_{v_{m}} &= S_{\pmb{\lambda}_{T}}\left( \Lambda_{m,k} e_{v_{(m+k)\mod \kappa}} + \sum_{i = 0}^{k-1} \Lambda_{m,i} \!\!\!\!\!\sum_{\substack{R\in\mathcal{R}\\i_{R} = i+m\mod\kappa}}\!\!\!\!\! \lambda_{\omega_{R}}S_{\pmb{\lambda}_{T}}^{k-i-1}e_{\omega_{R}} \right)\\
		&=\Lambda_{m,k+1}e_{v_{(m+k+1)}\mod \kappa}+\Lambda_{m,k}\sum_{\substack{R\in\mathcal{R}\\i_{R} = k+m\mod\kappa}}\!\!\!\!\! \lambda_{\omega_{R}}e_{\omega_{R}}\\
		&\qquad\qquad\qquad\qquad\qquad+\sum_{i = 0}^{k-1} \Lambda_{m,i} \!\!\!\!\!\sum_{\substack{R\in\mathcal{R}\\i_{R} = i+m\mod\kappa}}\!\!\!\!\! \lambda_{\omega_{R}}S_{\pmb{\lambda}_{T}}^{k-i}e_{\omega_{R}}\\
		&=\Lambda_{m,k+1} e_{v_{(m+k+1)\mod \kappa}} + \sum_{i = 0}^{k} \Lambda_{m,i} \!\!\!\!\!\sum_{\substack{R\in\mathcal{R}\\i_{R} = i+m\mod\kappa}}\!\!\!\!\! \lambda_{\omega_{R}}S_{\pmb{\lambda}_{T}}^{k-i}e_{\omega_{R}}.
	\end{align*}
	Hence, \eqref{FormNormOfSquareSmallPowers} holds for $ k+1 $.\\
	(iii). It is enough to apply the operator $ S^{n}_{\pmb{\lambda}_{T}} $ to both sides of \eqref{FormNormOfSquareSmallPowers} with $ k = \kappa $ and use (i).
\end{proof}

	\section{$ m $-isometric composition operators: graphs with one cycle}
\label{SecMIsometricCompOperators}
In \cite{kosmiderMisometricCompositionOperators2021} the authors studied composition operators with the property that the associated graphs have one cycle and one branching point. In this paper we also limit our considerations to operators such that related graphs have one cycle, but with no further assumptions on degrees of vertices. From now on we fix a discrete space $ (X,\mu) $ and a self-map $ T\!: X\to X $ such that $ h_{T}\in L^{\infty}(X) $, so that $ C_{T} $ (and consequently $ S_{\pmb{\lambda}_{T}} $) is bounded. Since, by Observation \ref{LemUnitaryEquivalence}, the composition operator $ C_{T} $ is unitarily equivalent to $ S_{\pmb{\lambda}_{T}} $, we investigate $ m $-isometricity of $ S_{\pmb{\lambda}_{T}} $. Whenever we assume that the associated graph $ G_{T} $ defined in Section \ref{SecClassificationOfCompOperators} satisfies Theorem \ref{ThmClassificationOfGraphs}.(ii), we stick to the notation introduced there.  \par
Let us start from the general result concerning $ m $-isometric composition operators on discrete spaces.
\begin{lemma}[\mbox{cf. \cite[Proposition 2.5]{kosmiderMisometricCompositionOperators2021}}]
	\label{LemGeneralCharacterizationOfMIsometries}
	Suppose $ T\!: X\to X $ is such that $ G_{T} $ is connected and such that $ h_{T}\in L^{\infty}(\mu) $. For $ m\in\mathbb{N}_{1} $ the following conditions are equivalent:
	\begin{enumerate}
		\item $ S_{\pmb{\lambda}_{T}} $ is $ m $-isometric,
		\item for every $ f\in \ell^{2}(V_{T}) $ there exists a polynomial $ p_{f}\in\mathbb{R}_{m-1}[x] $ satisfying
		\begin{equation*}
			p_{f}(n) = \lVert S_{\pmb{\lambda}_{T}}^{n}f\rVert^{2}, \qquad n\in\mathbb{N},
		\end{equation*} 
		\item for every $ v\in V_{T} $ there exists a polynomial $ p_{v}\in\mathbb{R}_{m-1}[x] $ satisfying
		\begin{equation}
			\label{FormSquareNormPolynomial}
			p_{v}(n) = \lVert S_{\pmb{\lambda}_{T}}^{n}e_{v}\rVert^{2}, \qquad n\in\mathbb{N}.
		\end{equation}
	\end{enumerate}
\end{lemma}
Before we state the proof we need the following observation.
\begin{observation}
	\label{ObsCoeffsEstimation}
	Suppose $ T\!: X\to X $ is such that $ G_{T} $ is connected and such that $ h_{T}\in L^{\infty}(\mu) $. Let $ m\in\mathbb{N}_{1} $. Let $ v\in V_{T} $ and suppose that $ p_{v}(x) = \sum_{i=0}^{m-1}a_{i}x^{i}\in \mathbb{R}_{m-1}[x] $ is polynomial satisfying \eqref{FormSquareNormPolynomial}. Then there exists a constant $ C \in [0,\infty) $ depending only on $ m $ and $ S_{\pmb{\lambda}_{T}} $ such that $ \lvert a_{i}\rvert \le C $ for $ i\in \mathbb{N}\cap[0,m-1] $.
\end{observation}
\begin{proof}
	First, note that the coefficient $ a_{0},\ldots,a_{m-1} $ form the unique solution of the following system of linear equations\footnote{We stick to the convention that $ 0^{0} = 1 $.}:
	\begin{equation*}
		\begin{bmatrix}
			0^{0} & 0^{1} & \ldots & 0^{m-1}\\
			1^{0} & 1^{1} & \ldots & 1^{m-1}\\
			\vdots & \vdots & \vdots & \vdots\\
			(m-1)^{0} & (m-1)^{1} & \ldots & (m-1)^{m-1}
		\end{bmatrix} \begin{bmatrix}
			a_{0}\\
			a_{1}\\
			\vdots\\
			a_{m-1}
		\end{bmatrix} = \begin{bmatrix}
			\lVert S_{\pmb{\lambda}_{T}}^{0}e_{v}\rVert^{2}\\
			\lVert S_{\pmb{\lambda}_{T}}^{1}e_{v}\rVert^{2}\\
			\vdots\\
			\lVert S_{\pmb{\lambda}_{T}}^{m-1}e_{v}\rVert^{2}\\
		\end{bmatrix}.
	\end{equation*}
	Applying Cramer's rule (see \cite[Theorem 3.36.(7)]{nairLinearAlgebra2018}) to the above system of linear equations, we obtain that there exist constants $ M_{\ell,j,m}\in \mathbb{R} $ ($ j,\ell\in \mathbb{N}\cap[0,m-1] $) depending on $ j,\ell $ (not on the operator $ S_{\pmb{\lambda}_{T}} $) satisfying
	\begin{equation*}
		a_{\ell} = \sum_{j=0}^{m-1} M_{\ell,j,m}\lVert S_{\pmb{\lambda}_{T}}^{j}e_{v}\rVert^{2}, \qquad \ell\in \mathbb{N}\cap[0,m-1].
	\end{equation*}
	Then, for every $ \ell\in \mathbb{N}\cap[0,m-1] $
	\begin{equation*}
		\lvert a_{\ell}\rvert \le \sum_{j=0}^{m-1} \lvert M_{\ell,j,m}\rvert\cdot \lVert S_{\pmb{\lambda}_{T}}^{j}e_{v}\rVert^{2} \le \sum_{j=0}^{m-1} \lvert M_{\ell,j,m}\rvert\cdot \lVert S_{\pmb{\lambda}_{T}}^{j}\rVert^{2}.
	\end{equation*}
	Hence, the constant
	\begin{equation*}
		C = \sup_{0\le \ell \le m-1} \sum_{j=0}^{m-1} \lvert M_{i,j,m}\rvert\cdot \lVert S_{\pmb{\lambda}_{T}}^{j}\rVert^{2}
	\end{equation*}
	satisfies $ \lvert a_{\ell}\rvert \le C $. Obviously, $ C $ depends only on $ m $ and $ S_{\pmb{\lambda}_{T}} $.
\end{proof}
\begin{proof}[Proof of Lemma \ref{LemGeneralCharacterizationOfMIsometries}]
	The equivalence of (i) and (ii) follows from remarks in \cite[p. 389]{aglerMisometricTransformationsHilbert1995}; the implication (ii)$ \implies $(iii) obviously holds. It remains to prove that (iii) implies (ii). Enumerate $ V_{T} = \{v_{j}\}_{j=0}^{N} $, where $ N\in \mathbb{N}\cup \{\infty\} $. Let $ f = \sum_{j=0}^{N} \alpha_{v_{j}}e_{v_{j}} \in \ell^{2}(V_{T}) $. Then, by Observation \ref{ObsPreservingOrthogonality},
	\begin{equation}
		\label{ProofFormSquareNormAnyVector}
		\lVert S_{\pmb{\lambda}_{T}}^{k}f\rVert^{2} = \sum_{j=0}^{N} \lvert \alpha_{v_{j}}\rvert^{2} \lVert S_{\pmb{\lambda}_{T}}^{k}e_{v_{j}}\rVert^{2}, \qquad k\in \mathbb{N}.
	\end{equation} For $ v\in V_{T} $ let $ p_{v} = \sum_{\ell=0}^{m-1}a_{\ell}^{v}x^{\ell} \in \mathbb{R}_{m-1}[x] $ be the polynomial satisfying \eqref{FormSquareNormPolynomial}. Let $ C $ be the constant given by Observation \ref{ObsCoeffsEstimation}. Then, for every $ \ell\in \mathbb{N}\cap[0,m-1] $,
	\begin{equation*}
		\sum_{j=0}^{N}\lvert \alpha_{v_{j}}\rvert^{2} \cdot \lvert a_{\ell}^{v_{j}}\rvert \le C\sum_{j=0}^{N}\lvert \alpha_{v_{j}}\rvert^{2}.
	\end{equation*}
	Hence, the series $ \sum_{j=0}^{N}\lvert \alpha_{v_{j}}\rvert^{2} \cdot a_{\ell}^{v_{j}} $ is absolutely convergent for every $ \ell\in \mathbb{N}\cap[0,m-1] $. By Lemma \ref{LemSeriesOfPolynomials},
	\begin{equation*}
		p_{f}(x) := \sum_{j=0}^{N} \lvert \alpha_{v_{j}}\rvert^{2}p_{v_{j}}(x) \in \mathbb{R}_{m-1}[x].
	\end{equation*}
	From \eqref{ProofFormSquareNormAnyVector} and \eqref{FormSquareNormPolynomial} it follows that $ p_{f}(n) = \lVert S_{\pmb{\lambda}_{T}}^{n}f\rVert^{2} $ for every $ n\in \mathbb{N} $.
\end{proof}
In the next results we characterize $ m $-isometric composition operators such that the graph $ G_{T} $ has one cycle. Let us begin with the following lemma.
\begin{lemma}
	\label{LemPolynomialOneStepBack}
 	Suppose $ T\!: X\to X $ is such that $ G_{T} $ is connected and satisfies Theorem \ref{ThmClassificationOfGraphs}.(ii). Assume $ h_{T}\in L^{\infty}(\mu) $ and let $ m\in\mathbb{N}_{2} $. Suppose that for every $ R\in \mathcal{R} $ there exists a polynomial $ p_{\omega_{R}}\in \mathbb{R}_{m-1}[x] $ satisfying \eqref{FormSquareNormPolynomial} with $ v = \omega_{R} $. Then for every $ k\in \mathbb{N}\cap[0,\kappa-1] $ the formula
	\begin{equation*}
		q_{k}(x):= \sum_{\substack{R\in \mathcal{R}\\i_{R}=k}} \lambda_{\omega_{R}}^{2} p_{\omega_{R}}(x), \qquad x\in \mathbb{R},
	\end{equation*} 
	defines a polynomial of degree at most $ m-1 $. Moreover,
	\begin{equation}
		\label{FormMaxDegree}
		\deg q_{k} = \max\{\deg p_{\omega_{R}}\!: R\in \mathcal{R}, \ i_{R} = k\}, \qquad k\in \mathbb{N}\cap[0,\kappa-1].
	\end{equation}
\end{lemma}
\begin{proof}
	Let $ k\in \mathbb{N}\cap[0,\kappa-1] $. Enumerate $ \{R\in \mathcal{R}\!: i_{R} = k\} = \{R_{j}\}_{j=0}^{N} $, where $ N\in \mathbb{N}\cup\{\infty\} $ is the cardinality of the set on the left hand side. Write $ p_{\omega_{R}}(x) = \sum_{i=0}^{m-1}a_{i}^{R}x^{i} $, $ x \in \mathbb{R} $, $ R\in \mathcal{R} $. We will check that the series
	\begin{equation}
		\label{ProoFormCoeffsSeries}
		\sum_{j=0}^{N} \lambda_{\omega_{R_{j}}}^{2} a_{i}^{R_{j}}
	\end{equation}
	is absolutely convergent for every $ i\in \mathbb{N}\cap[0,m-1] $. Since, by Observation \ref{ObsPreservingOrthogonality}, for every $ \ell\in \mathbb{N}\cap[0,m-1] $,
	\begin{equation*}
		\lVert S_{\pmb{\lambda}_{T}}e_{v_{k}}\rVert^{2} = \lambda_{k+1\mod \kappa}^{2} + \sum_{j=0}^{N}\lambda_{\omega_{R_{j}}}^{2},
	\end{equation*}
	it follows that the series $ \sum_{j=0}^{N}\lambda_{\omega_{R_{j}}}^{2} $ is convergent. Using Observation \ref{ObsCoeffsEstimation}, we obtain that the series $ \sum_{j=0}^{N} \lambda_{\omega_{R_{j}}}^{2} a_{i}^{R_{j}} $ is absolutely convergent. By Lemma \ref{LemSeriesOfPolynomials}, $ q_{k} \in \mathbb{R}_{m-1}[x] $. Since $ p_{\omega_{R}}(n) \ge 0 $ for every $ n\in \mathbb{N} $, the leading coefficient of $ p_{\omega_{R}} $ is non-negative for every $ R\in \mathcal{R} $. Hence, \eqref{FormMaxDegree} holds. This completes the proof.
\end{proof}
It turns out that if $ S_{\pmb{\lambda}_{T}} $ is $ m $-isometric, then on 'tree parts' of $ G_{T} $ it is actually $ (m-1) $-isometric. The similar result is presented in \cite[Theorem 2.10]{kosmiderMisometricCompositionOperators2021} for composition operators such that the graph $ G_{T} $ has one cycle and one branching point; however, there is no clear way how to adjust the proof in a way to obtain the same conclusion in a more general situation. In our proof we use a different approach to prove this result in full generality.
\begin{lemma}
	\label{LemOnTreesM-1Isometric}
	Suppose $ T\!: X\to X $ is such that $ G_{T} $ is connected and satisfies the condition as in Theorem \ref{ThmClassificationOfGraphs}.(ii). Assume $ h_{T}\in L^{\infty}(\mu) $ and let $ m\in\mathbb{N}_{2} $. Suppose $ S_{\pmb{\lambda}_{T}}\in \mathbf{B}(\ell^{2}(V_{T})) $ is $ m $-isometric. Then for every $ v\in V_{T}\setminus\{v_{0},\ldots,v_{\kappa-1}\} $ the polynomial $ p_{v} $ satisfying \eqref{FormSquareNormPolynomial} is of degree at most $ m-2 $.
\end{lemma}
\begin{proof}
	First, we prove that for every $ R\in \mathcal{R} $, the polynomial $ p_{\omega_{R}} $ satisfying \eqref{FormSquareNormPolynomial} with $ v = \omega_{R} $ is of degree at most $ m-2 $. By Lemma \ref{LemGeneralCharacterizationOfMIsometries}, there exist $ p_{v_{0}}\in\mathbb{R}_{m-1}[x] $ and $ p_{\omega_{R}}\in\mathbb{R}_{m-1}[x] $ ($ R\in\mathcal{R} $) satisfying \eqref{FormSquareNormPolynomial}. Combining \eqref{FormSquareNormPolynomial} and \eqref{FormRecursiveNormOfSquare} we obtain that for every $ n\in \mathbb{N} $,
	\begin{equation}
		\label{FormRecursivePolynomialOnCycle}
		p_{v_{0}}(n+\kappa)-p_{v_{0}}(n) = \sum_{i = 0}^{\kappa-1} \Lambda_{0,i}^{2}\sum_{\substack{R\in\mathcal{R}\\i_{R} = i}} \lambda_{\omega_{R}}^{2}p_{\omega_{R}}(n+\kappa-i-1).
	\end{equation}
	From Observation \ref{ObsLoweringDegree} we deduce that the polynomial $ q(\,\cdot\,) = p_{v_{0}}(\,\cdot\, + \kappa) - p_{v_{0}}(\,\cdot\,) $ is of degree less than or equal to $ m-2 $. By Lemma \ref{LemPolynomialOneStepBack}, the formula
	\begin{equation*}
		r(x) = \sum_{i = 0}^{\kappa-1} \Lambda_{0,i}^{2}\sum_{\substack{R\in\mathcal{R}\\i_{R} = i}} \lambda_{\omega_{R}}^{2}p_{\omega_{R}}(x+\kappa-i-1), \qquad x\in \mathbb{R}.
	\end{equation*}
	defines a polynomial of degree at most $ m-1 $. We infer from \eqref{FormRecursivePolynomialOnCycle} that $ q = r $ and, consequently, $ r \in \mathbb{R}_{m-2}[x] $. From Lemma \ref{LemPolynomialOneStepBack} it follows that $ p_{\omega_{R}} \in \mathbb{R}_{m-2}[x] $. Next, if $ R\in \mathcal{R} $ and $ v\in \mathcal{R} $ is such that $ p_{v}\in \mathbb{R}_{m-2}[x] $, then proceeding as in the proof of Lemma \ref{LemPolynomialOneStepBack}, we obtain that
	\begin{equation*}
		q_{v}(x) := \sum_{\substack{u\in V_{T}}\\(v,u)\in E_{T}} \lambda_{u}^{2}p_{u}(x-1), \qquad x\in \mathbb{R},
	\end{equation*}
	is the polynomial of degree at most $ m-1 $. For every $ n\in \mathbb{N}_{1} $,
	\begin{equation*}
		q_{v}(n) = \sum_{\substack{u\in V_{T}}\\(v,u)\in E_{T}} \lambda_{u}^{2}p_{u}(n-1) = \sum_{\substack{u\in V_{T}}\\(v,u)\in E_{T}} \lambda_{u}^{2}\lVert S_{\pmb{\lambda}_{T}}^{n-1}e_{u}\rVert^{2} = \lVert S_{\pmb{\lambda}_{T}}^{n}e_{v}\rVert^{2}.
	\end{equation*}
	Hence, $ q_{v} = p_{v} $. Since for every $ u\in V_{T} $ the leading coefficient of $ p_{u} $ is non-negative, it follows that $ \deg p_{u} \le \deg p_{v} $. Therefore, $ p_{u}\in \mathbb{R}_{m-2}[x] $. The application of Lemma \ref{LemInductionOnTree} completes the proof.
\end{proof}
Let us introduce the following notation: for $ m\in \mathbb{N}_{2} $ and $ \kappa\in \mathbb{N}_{1} $ we set
\begin{equation*}
	A_{m,\kappa} = \begin{bmatrix}
		d_{1,1} & 0 & 0 & \ldots & 0 & 0\\
		d_{1,2} & d_{2,1} & 0 & \ldots & 0 & 0 \\
		\vdots & \vdots & \vdots & \vdots & \vdots \\
		d_{1,m-1} & d_{2,m-2} & d_{3,m-3} & \ldots & d_{m-1,1} & 0\\
		0 & 0 & 0 & \ldots & 0 & 1
		\end{bmatrix},
\end{equation*}
where $ d_{i,j} = \binom{m-i}{m-j-i}\kappa^{j} $, $ i,j\in \mathbb{N} $, $ j+i\le m $, and if $ \kappa > 1 $,
\begin{equation*}
	B_{m,\kappa} = \begin{bmatrix}
		1^{m-1} & 1^{m-2} & \ldots & 1^{0}\\
		2^{m-1} & 2^{m-2} & \ldots & 2^{0}\\
		\vdots & \vdots & \vdots & \vdots\\
		(\kappa-1)^{m-1} & (\kappa-1)^{m-2} & \ldots & (\kappa-1)^{0} 
		\end{bmatrix}.
\end{equation*}
Now we are in the position to state the main result of this section.
\begin{theorem}
	\label{ThmMIsometricCompositionOperatorsOneCycle}
	Suppose $ T\!: X\to X $ is such that $ G_{T} $ is connected and satisfies Theorem \ref{ThmClassificationOfGraphs}.(ii). Assume $ h_{T}\in L^{\infty}(\mu) $. For $ m\in\mathbb{N}_{2} $ the following conditions are equivalent:
	\begin{enumerate}
		\item $ S_{\pmb{\lambda}_{T}} $ is $ m $-isometric,
		\item for every $ v\in V_{T}\setminus\{v_0,\ldots,v_{\kappa-1} \} $ there exists the unique polynomial $ p_{v}\in \mathbb{R}_{m-2}[x] $ satisfying \eqref{FormSquareNormPolynomial} and if $ \kappa > 1 $, $ m = \rank \tilde{A}_{m,\kappa} = \rank\tilde{B}_{m,\kappa} $, where
		\begin{equation*}
			\tilde{A}_{m,\kappa} = \begin{bmatrix}
				A_{m,\kappa} & b
			\end{bmatrix}, \tilde{B}_{m,\kappa} = \begin{bmatrix}
				A_{m,\kappa} & b\\
				B_{m,\kappa} & a
			\end{bmatrix},
		\end{equation*}
		\begin{equation*}
			b = \begin{bmatrix}
				b_{m-2} & b_{m-3} & \ldots & b_{0} & 1
			\end{bmatrix}^{T}, \quad  a = \begin{bmatrix}
				a_{1} & a_{2} & \ldots & a_{\kappa-1}
			\end{bmatrix}^{T},
		\end{equation*}
		$ a_{i} = \lVert S_{\pmb{\lambda}}^{i}e_{v_{0}}\rVert^{2} $ ($ i\in\mathbb{N} $, $ j\in\mathbb{N}\cap[1,m-i-1] $), and $ q(x) = \sum_{i=0}^{m-2} b_{i}x^{i} \in \mathbb{R}_{m-2}[x] $ is the polynomial defined by the formula\footnote{The fact that this formula defines a polynomial of degree at most $ m-2 $ follows from Lemma \ref{LemPolynomialOneStepBack}.}
		\begin{equation}
			\label{FormPolynomialP0MinusP0C}
			q(x) = \sum_{i = 0}^{\kappa-1} \Lambda_{0,i}^{2}\sum_{\substack{R\in\mathcal{R}\\i_{R} = i}} \lambda_{\omega_{R}}^{2}p_{\omega_{R}}(x+\kappa-i-1), \qquad x\in \mathbb{R}.
		\end{equation}
	\end{enumerate}
	Moreover, if (ii) holds, then $ p_{v_{0}}(x) = \sum_{j = 0}^{m-1} c_{j}x^{j} $, where $ c = (c_{m-1},\ldots,c_{0}) $ is the solution of 
	\begin{align}
		\label{FormLinearSystemRecursion}
		A_{m,\kappa}c = b
	\end{align}
\end{theorem}
\begin{proof}
	(i) $ \Longrightarrow $ (ii). By Lemma \ref{LemGeneralCharacterizationOfMIsometries}, for every $ v\in V_{T} $ there exists $ p_{v}\in\mathbb{R}_{m-1}[x] $ satisfying \eqref{FormSquareNormPolynomial}. By Lemma \ref{LemOnTreesM-1Isometric}, for every $ R\in\mathcal{R} $ and every $ v\in R $, $ p_{v}\in\mathbb{R}_{m-2}[x] $. Next, writing $ p_{v_{0}}(x) = \sum_{j = 0}^{m-1} c_{j}x^{j} $ and using the binomial formula, we have
	\begin{equation*}
		p_{v_{0}}(x+\kappa)-p_{v_{0}}(x) = \sum_{\ell = 0}^{m-2}\left(\sum_{j = \ell+1}^{m-1}\binom{j}{\ell}c_{j}\kappa^{j-\ell} \right)x^{j}.
	\end{equation*}
	Since, by Lemma \ref{LemPolynomialOneStepBack}, the formula
	\begin{equation*}
		q(x) = \sum_{i = 0}^{\kappa-1} \Lambda_{0,i}^{2}\sum_{\substack{R\in\mathcal{R}\\i_{R} = i}} \lambda_{\omega_{R}}^{2}p_{\omega_{R}}(x+\kappa-i-1)
	\end{equation*}
	defines a polynomial of degree at most $ m-2 $ satisfying, by \eqref{FormRecursiveNormOfSquare},
	\begin{equation*}
		q(n) = p_{v_{0}}(n+\kappa)-p_{v_{0}}(n), \qquad n\in\mathbb{N},
	\end{equation*}
	writing $ q(x) = \sum_{\ell = 0}^{m-2}b_{\ell}x^{\ell} $ we obtain
	\begin{equation}
		\label{FormCoeffsOfpv0}
		b_{\ell} = \sum_{j = \ell+1}^{m-1}\binom{j}{\ell}c_{j}\kappa^{j-\ell}, \qquad \ell\in\mathbb{N}\cap[0,m-2].
	\end{equation}
	Hence, $ (c_{m-1},\ldots,c_{0}) $ satisfies \eqref{FormLinearSystemRecursion}. Moreover, since the determinant of $ A_{m,\kappa} $ is non-zero, we infer from Cramer's rule that $ c = (c_{m-1},\ldots,c_{0}) $ is the only solution of this system. By the Kronecker-Capelli theorem (see \cite[Theorem 3.36.(1)-(3)]{nairLinearAlgebra2018}) this implies that $ m = \rank\tilde{A}_{m,\kappa} $. Assume now $ \kappa > 1 $. From the equality $ p_{v_{0}}(i) = a_{i} $ ($ i\in\mathbb{N}\cap[0,\kappa-1] $) we get that $ c $ is also the solution of the following system of linear equations:
	\begin{align*}
		b_{\ell} &= \sum_{j = \ell+1}^{m-1}\binom{j}{\ell}c_{j}\kappa^{j-\ell}, \qquad \ell\in\mathbb{N}\cap[0,m-2],\\
		a_{\ell} &= \sum_{j=0}^{m-1}c_{j}\ell^{j}, \qquad \ell\in \mathbb{N}\cap[1,\kappa-1];
	\end{align*}
	the matrix representation of this system takes the form:
	\begin{align}
		\label{FormLinearSystemRecursionSmallPowers}
		\begin{bmatrix}
			A_{m,\kappa}\\
			B_{m,\kappa}
		\end{bmatrix} c = \begin{bmatrix}
			b\\
			a
		\end{bmatrix}
	\end{align}
	Again, by the Kronecker-Capelli theorem, $ m = \rank\tilde{B}_{m,\kappa} $, which gives (ii).\\
	(ii) $ \Longrightarrow $ (i). By Lemma \ref{LemGeneralCharacterizationOfMIsometries}, it is enough to find the polynomials $ p_{v_{k}}\in\mathbb{R}[x] $ ($ k\in\mathbb{N}\cap[0,\kappa-1] $) satisfying \eqref{FormSquareNormPolynomial}. First, observe that for every $ k\in\mathbb{N}\cap[1,\kappa-1] $ the polynomial $ p_{v_{k}} $ is uniquely determined by $ p_{v_{0}} $ and $ p_{v} $ ($ v\in V\setminus\{v_{0},\ldots,v_{\kappa-1} \} $). Indeed, assuming we have $ p_{v_{0}} $ satisfying \eqref{FormSquareNormPolynomial} and using \eqref{FormNormOfSquareSmallPowers}, we obtain
	\begin{align*}
		\lVert S_{\pmb{\lambda}_{T}}^{n}e_{v_{k}}\rVert^{2} &= \Lambda_{0,k}^{-2} (p_{v_{0}}(n+k) \\
		&- \sum_{i = 0}^{k-1} \Lambda_{0,i}^{2}\sum_{\substack{R\in\mathcal{R}\\i_{R} = i\mod\kappa}} \lambda_{\omega_{R}}^{2}p_{\omega_{R}}(n+k-i-1) ) , \quad n\in \mathbb{N}.
	\end{align*}
	By Lemma \ref{LemPolynomialOneStepBack},
	\begin{align*}
		p_{v_{k}}(\,\cdot\,) &:= \Lambda_{0,k}^{-2} (p_{v_{0}}(\,\cdot\,+k) - \sum_{i = 0}^{k-1} \Lambda_{0,i}^{2}\sum_{\substack{R\in\mathcal{R}\\i_{R} = i\mod\kappa}} \lambda_{\omega_{R}}^{2}p_{\omega_{R}}(\,\cdot+k-i-1) )
	\end{align*}
	is a polynomial of degree at most $ m-1 $, which satisfies \eqref{FormSquareNormPolynomial}. Hence, we have to find only the polynomial $ p_{v_{0}} $. By Cramer's rule, $ \eqref{FormLinearSystemRecursion} $ has only one solution; call $ c = (c_{m-1},\ldots,c_{0}) $ this unique solution and define $ p_{v_{0}}(x) = \sum_{j=0}^{m-1}c_{j}x^{j} $ for $ x\in \mathbb{R} $. By (ii) and the Kronecker-Capelli theorem, $ (c_{m-1},\ldots,c_{0}) $ has to satisfy also \eqref{FormLinearSystemRecursionSmallPowers}. Therefore,
	\begin{equation*}
		p_{v_{0}}(n) = \sum_{j=0}^{m-1}c_{j}n^{j} = a_{j} = \lVert S_{\pmb{\lambda}_{T}}^{n}e_{v_{0}}\rVert^{2}, \qquad n\in\mathbb{N}\cap[0,\kappa-1].
	\end{equation*}
	Since $ (c_{m-1},\ldots,c_{0}) $ satisfies \eqref{FormCoeffsOfpv0}, we get that $ p_{v_{0}}(\,\cdot\,+\kappa)-p_{v_{0}}(\,\cdot\,) = q(\,\cdot\,) $. Thus, from \eqref{FormRecursiveNormOfSquare} we obtain that $ p_{v_{0}}(n) = \lVert S_{\pmb{\lambda}_{T}}^{n}e_{v_{0}}\rVert^{2} $ for every $ n\in\mathbb{N} $. The 'moreover' part easily follows from the above reasoning.
\end{proof}
In \cite{kosmiderMisometricCompositionOperators2021} the authors proved that in the class of composition operators such that the graph $ G_{T} $ has one cycle and one branching point every completely hyperexpansive operator is 2-isometric. Using different approach, with the aid of Theorem \ref{ThmMIsometricCompositionOperatorsOneCycle}, we are able to prove this result in full generality.
\begin{theorem}[\mbox{cf. \cite[Corollary 2.15]{kosmiderMisometricCompositionOperators2021}}]
Suppose $ T\!: X\to X $ is such that $ G_{T} $ is connected and satisfies Theorem \ref{ThmClassificationOfGraphs}.(ii). Suppose $ h_{T}\in L^{\infty}(\mu) $. The following conditions are equivalent:
	\begin{enumerate}
		\item $ S_{\pmb{\lambda}_{T}} $ is 2-isometric,
		\item $ S_{\pmb{\lambda}_{T}} $ is completely hyperexpansive.
	\end{enumerate}
\end{theorem}
\begin{proof}
	(i)$ \Longrightarrow $(i). By Lemma \ref{LemGeneralCharacterizationOfMIsometries}, if $ S_{\pmb{\lambda}_{T}} $ is 2-isometric, then it is $ m $-isometric for every $ m\in \mathbb{N}_{2} $. Combining this with \cite[Proposition 1.5]{aglerMisometricTransformationsHilbert1995} we obtain that $ S_{\pmb{\lambda}_{T}} $ is completely hyperexpansive.\\
	(ii)$ \Longrightarrow $(i). From (ii), by \cite[Remark 2]{athavaleCompletelyHyperexpansiveOperators1996}, it follows that for every $ v\in V_{T} $ there exists the measure $ \tau_{v}\!: \mathcal{B}([0,1])\to [0,\infty) $ satisfying
	\begin{equation}
		\label{FormCompletelyAlternatingAtVertex}
		\lVert S_{\pmb{\lambda}_{T}}^{n}e_{v}\rVert^{2} = 1+\int_{[0,1]}(1+t+\ldots+t^{n-1})\ddd \tau_{v}(t), \qquad n\in\mathbb{N}_{1}.
	\end{equation}
	Inserting \eqref{FormCompletelyAlternatingAtVertex} into \eqref{FormRecursiveNormOfSquare} we obtain that for every $ n\in\mathbb{N}_{1} $,
	\begin{align}
		\label{FormCompletelyAlternatingRecursiveOnCycle}
		&1+\int_{[0,1]}(1+\ldots+t^{n+\kappa-1})\ddd\tau_{v_{m}}(t) = 1+\int_{[0,1]}(1+\ldots+t^{n-1})\ddd\tau_{v_{m}}(t)\nonumber\\
		&+\sum_{i=0}^{\kappa-1} \Lambda_{m,i}^{2}\!\!\!\!\!\sum_{\substack{R\in\mathcal{R}\\i_{R} = i+m\mod\kappa}}\!\!\!\!\!\lambda_{\omega_{R}}^{2}(1+\int_{[0,1]}(1+\ldots+t^{n+\kappa-i-2})\ddd\tau_{\omega_{R}}(t))
	\end{align}
	If $ v\in V_{T} $, then for every $ n\in \mathbb{N} $ we have
	\begin{equation*}
		\lambda_{u}^{2}\int_{[0,1]}t^{n}\ddd \tau_{u}(t) \le \lambda_{u}^{2}(1+\int_{[0,1]}(1+\ldots+t^{n}))\ddd\tau_{u}(t), \qquad u\in V_{T}, (v,u)\in E_{T}.
	\end{equation*}
	By \eqref{FormCompletelyAlternatingAtVertex} and Observation \ref{ObsPreservingOrthogonality}, for every $ n\in \mathbb{N} $,
	\begin{equation*}
		\sum_{\substack{u\in V_{T}\\(v,u)\in E_{T}}} \lambda_{u}^{2}(1+\int_{[0,1]}(1+\ldots+t^{n}))\ddd\tau_{u}(t) = \sum_{\substack{u\in V_{T}\\(v,u)\in E_{T}}}\!\!\!\! \lambda_{u}^{2} \lVert S_{\pmb{\lambda}_{T}}^{n+1}e_{u}\rVert^{2} = \lVert S_{\pmb{\lambda}_{T}}^{n+2}e_{v}\rVert^{2}.
	\end{equation*}
	Thus, the series
	\begin{equation*}
		\sum_{\substack{u\in V_{T}\\(v,u)\in E_{T}}}\lambda_{u}^{2}\int_{[0,1]}t^{n}\ddd\tau_{u}
	\end{equation*}
	is convergent for every $ v\in V_{T} $ and $ n\in \mathbb{N} $. By \eqref{FormCompletelyAlternatingAtVertex} we have that
	\begin{align*}
		&\int_{[0,1]}t^{n+\kappa}\ddd\tau_{v_{m}}(t) = \lVert S_{\pmb{\lambda}_{T}}^{n+\kappa+1}e_{v_{m}}\rVert^{2}-\lVert S_{\pmb{\lambda}_{T}}^{n+\kappa}e_{v_{m}}\rVert^{2}\\
		&=(1+\int_{[0,1]}(1+\ldots+t^{n+\kappa})\ddd\tau_{v_{m}}(t)) - (1+\int_{[0,1]}(1+\ldots+t^{n+\kappa-1})\ddd\tau_{v_{m}}(t))
	\end{align*} 
	Using \eqref{FormCompletelyAlternatingRecursiveOnCycle} with $ n $ replaced with $ n+1 $, we get that for every $ n\in\mathbb{N}_{1} $,
	\begin{align*}
		\int_{[0,1]}t^{n+\kappa}\ddd\tau_{v_{m}}(t) &= \int_{[0,1]}t^{n}\ddd\tau_{v_{m}}(t) \\
		&+ \sum_{i=0}^{\kappa-1} \Lambda_{m,i}^{2}\!\!\!\!\!\!\!\!\sum_{\substack{R\in\mathcal{R}\\i_{R} = i+m\mod\kappa}} \!\!\!\!\!\!\!\lambda_{\omega_{R}}^{2}\int_{[0,1]}t^{n+\kappa-i-1}\ddd\tau_{\omega_{R}}(t),
	\end{align*}
	which implies that
	\begin{align*}
		\int_{[0,1]}t^{n}(t^{\kappa}-1)\ddd\tau_{v_{m}}(t) = \sum_{i=0}^{\kappa-1} \Lambda_{m,i}^{2}\!\!\!\!\!\sum_{\substack{R\in\mathcal{R}\\i_{R} = i+m\mod\kappa}}\!\!\!\!\! \lambda_{\omega_{R}}^{2}\int_{[0,1]}t^{n+\kappa-i-1}\ddd\tau_{\omega_{R}}(t).
	\end{align*}
	Since $ t^{\kappa} -1 \le 0 $ for every $ t\in[0,1] $, the integral on the left hand side is non-positive. On the other hand, we have $ \lambda_{v}>0 $ for every $ v\in V_{T} $, so each term in the sum on the right hand side is non-negative. Therefore, we get that for all $ m\in\mathbb{N}\cap[0,\kappa-1] $, $ n\in\mathbb{N}_{1} $,
	\begin{equation}
		\label{FormMeasuresOnCycle}
		\int_{[0,1]} t^{n}(t^{\kappa}-1)\ddd\tau_{v_{m}}(t) = 0
	\end{equation}
	and for all $ R\in \mathcal{R} $,
	\begin{equation}
		\label{FormMeasuresOnTrees}
		\int_{[0,1]} t^{n+\kappa-i-1}\ddd\tau_{\omega_{R}}(t) = 0, \qquad i\in\mathbb{N}\cap[0,\kappa-1],\ i+m= i_{R}\mod\kappa.
	\end{equation}
	By \eqref{FormMeasuresOnTrees}, $ \tau_{\omega_{R}}((0,1]) = 0 $. This implies that for all $ R\in \mathcal{R} $, $ \tau_{\omega_{R}} = C_{R}\delta_{0} $ with some $ C_{R}\in [0,\infty) $. In turn, by \eqref{FormMeasuresOnCycle}, $ \tau_{v_{m}}((0,1)) = 0 $. Hence, for all $ m\in \mathbb{N}\cap[0,\kappa-1] $, $ \tau_{v_{m}} = C_{m}\delta_{0}+D_{m}\delta_{1} $ with some constants $ C_{m},D_{m}\in [0,\infty) $. Inserting these into \eqref{FormCompletelyAlternatingRecursiveOnCycle} we obtain that for every $ n\in\mathbb{N}_{2} $ and every $ m\in\mathbb{N}\cap[0,\kappa-1] $,
	\begin{equation*}
		1+C_{m}+(n+\kappa)D_{m} = 1+C_{m}+nD_{m}+\sum_{i=0}^{\kappa-1}\Lambda_{m,i}^{2}\!\!\!\!\!\!\!\sum_{\substack{R\in\mathcal{R}\\i_{R} = i+m\mod\kappa}}\!\!\!\!\!\!\!\lambda_{\omega_{R}}^{2}(1+C_{R}).
	\end{equation*}
	Hence,
	\begin{equation}
		\label{FormMassAtZeroMeasuresOnCycle}
		\kappa D_{m} = \sum_{i=0}^{\kappa-1}\Lambda_{m,i}^{2}\!\!\!\!\!\!\!\sum_{\substack{R\in\mathcal{R}\\i_{R} = i+m\mod\kappa}}\!\!\!\!\!\!\!\lambda_{\omega_{R}}^{2}(1+C_{R}).
	\end{equation}
	If $ \kappa > 1 $, then, using \eqref{FormRecursiveNormOfSquare} with $ n = 0 $ and \eqref{FormCompletelyAlternatingAtVertex} with $ v=v_{m} $, we get that
	\begin{align*}
		&1+C_{m}+\kappa D_{m} \\
		&= 1+ \sum_{i=0}^{\kappa-2}\Lambda_{m,i}^{2}\!\!\!\!\!\!\!\sum_{\substack{R\in\mathcal{R}\\i_{R} = i+m\mod\kappa}}\!\!\!\!\!\!\!\lambda_{\omega_{R}}^{2}(1+C_{R}) + \Lambda_{m,\kappa-1}^{2}\sum_{\substack{R\in\mathcal{R}\\i_{R} = m-1}} \lambda_{\omega_{R}}^{2}\\
		&= 1+ \sum_{i=0}^{\kappa-1}\Lambda_{m,i}^{2}\!\!\!\!\!\!\!\sum_{\substack{R\in\mathcal{R}\\i_{R} = i+m\mod\kappa}}\!\!\!\!\!\!\!\lambda_{\omega_{R}}^{2}(1+C_{R}) - \Lambda_{m,\kappa-1}^{2}\sum_{\substack{R\in\mathcal{R}\\i_{R} = m-1}} \lambda_{\omega_{R}}^{2}C_{R};
	\end{align*}
	if $ \kappa =1 $, by \eqref{FormNormOfSquareSmallPowers} applied with $ k=1 $, we have
	\begin{align*}
		&1+C_{0}+\kappa D_{0} \\
		&= 1+ \sum_{\substack{R\in\mathcal{R}}} \lambda_{\omega_{R}}^{2}\\
		&= 1+ \sum_{\substack{R\in\mathcal{R}\\}}\lambda_{\omega_{R}}^{2}(1+C_{R}) - \sum_{\substack{R\in\mathcal{R}}} \lambda_{\omega_{R}}^{2}C_{R};
	\end{align*}
	Combining the above with \eqref{FormMassAtZeroMeasuresOnCycle}, we obtain
	\begin{align*}
		1+C_{m}+\kappa D_{m} = 1+ \kappa D_{m} - \Lambda_{m,\kappa-1}^{2}\sum_{\substack{R\in\mathcal{R}\\i_{R} = m-1}} \lambda_{\omega_{R}}^{2}C_{R},
	\end{align*}
	which implies that
	\begin{align*}
		C_{m} = -\Lambda_{m,m+\kappa-1}^{2}\sum_{\substack{R\in\mathcal{R}\\i_{R} = m-1}} \lambda_{\omega_{R}}^{2}C_{R}.
	\end{align*}
	Since $ \lambda_{v}>0 $ for all $ v\in V_{T} $, we have that $ C_{m} \le 0 $ for $ m\in\mathbb{N}\cap[0,\kappa-1] $. On the other hand, $ C_{m}\ge 0 $. Therefore, $ C_{m} = 0 $ for all $ m\in\mathbb{N}\cap[0,\kappa-1] $, and, consequently, $ C_{R} = 0 $ for every $ R\in \mathcal{R} $. Hence,
	\begin{equation*}
		\tau_{\omega_{R}} = 0, \qquad R\in \mathcal{R},
	\end{equation*}
	and
	\begin{equation*}
		\tau_{v_{m}} = D_{m}\delta_{1}, \qquad m\in \mathbb{N}\cap[0,\kappa-1].
	\end{equation*}
	By \eqref{FormCompletelyAlternatingAtVertex}, this implies that $ \lVert S_{\pmb{\lambda}_{T}}^{n}e_{\omega_{R}}\rVert = 1 $ for $ n\in\mathbb{N} $ and that $ \lVert S_{\pmb{\lambda}}^{n}e_{v_{m}}\rVert^{2} = 1+nD_{m} $ for $ n\in\mathbb{N} $. It remains to show that
	\begin{equation*}
		\tau_{v} = 0, \qquad v\in V_{T}\setminus\{v_{0},\ldots,v_{\kappa-1} \}, n\in\mathbb{N},
	\end{equation*}
	which implies that
	\begin{equation*}
		\lVert S_{\pmb{\lambda}_{T}}^{n}e_{v}\rVert = 1, \qquad v\in V_{T}\setminus\{v_{0},\ldots,v_{\kappa-1} \}, n\in\mathbb{N}.
	\end{equation*}
	Since we have already proved this for $ v = \omega_{R} $, $ R\in\mathcal{R} $, by Lemmata \ref{LemAfterRemovingSCC} and \ref{LemInductionOnTree}, it is enough to show that if $ \tau_{v} =0  $ for some $ v\in V_{T}\setminus\{v_{0},\ldots,v_{\kappa-1} \} $, then $ \tau_{u} = 0 $ for every $ u\in V_{T} $ satisfying $ (v,u)\in E_{T} $. Assume $ v\in V_{T}\setminus\{v_{0},\ldots,v_{\kappa-1} \} $ is such that $ \tau_{v} = 0 $. Since
	\begin{equation*}
		S_{\pmb{\lambda}_{T}}^{n+1} e_{v} = \sum_{\substack{u\in V_{T}\\ (v,u)\in E_{T}}} \lambda_{u}S_{\pmb{\lambda}_{T}}^{n}e_{u}, \qquad n\in\mathbb{N},
	\end{equation*}
	it follows from \eqref{FormCompletelyAlternatingAtVertex} that
	\begin{equation}
		\label{FormCompletelyAlternatingVertexWithZeroMeasure}
		1 = \sum_{\substack{u\in V_{T}\\ (v,u)\in E_{T}}} \lambda_{u}^{2}(1+\int_{[0,1]}(1+\ldots+t^{n-1} )\ddd\tau_{u}(t)), \qquad n\in\mathbb{N}_{1}.
	\end{equation}
	If we replace $ n $ with $ n+1 $ in the above equality, we obtain
	\begin{align*}
		1 &= \sum_{\substack{u\in V_{T}\\ (v,u)\in E_{T}}} \lambda_{u}^{2}(1+\int_{[0,1]}(1+\ldots+t^{n} )\ddd\tau_{u}(t))\\
		&= \sum_{\substack{u\in V_{T}\\ (v,u)\in E_{T}}} \lambda_{u}^{2}(1+\int_{[0,1]}(1+\ldots+t^{n-1} )\ddd\tau_{u}(t))+\sum_{\substack{u\in V_{T}\\ (v,u)\in E_{T}}} \lambda_{v}^{2}\int_{[0,1]} t^{n}\ddd\tau_{u}(t)\\
		&\stackrel{\eqref{FormCompletelyAlternatingVertexWithZeroMeasure}}{=}1+ \sum_{\substack{u\in V_{T}\\ (v,u)\in E_{T}}} \lambda_{v}^{2}\int_{[0,1]} t^{n}\ddd\tau_{u}(t).
	\end{align*}
	Hence,
	\begin{equation*}
		0 = \sum_{\substack{u\in V_{T}\\ (v,u)\in E_{T}}} \lambda_{v}^{2}\int_{[0,1]} t^{n}\ddd\tau_{u}(t), \qquad n\in\mathbb{N}_{1},
	\end{equation*}
	which implies that for every $ u\in V_{T} $ satisfying $ (v,u)\in E_{T} $ we have $ \tau_{u}((0,1]) = 0 $. Thus, $ \tau_{u} = C_{u}\delta_{0} $ with some $ C_{u}\in [0,\infty)$. Then \eqref{FormCompletelyAlternatingVertexWithZeroMeasure} takes the form
	\begin{equation*}
		1 = \sum_{\substack{u\in V_{T}\\ (v,u)\in E_{T}}} \lambda_{u}^{2}(1+C_{u}), \qquad n\in\mathbb{N}_{1}.
	\end{equation*}
	But
	\begin{equation*}
		1 = \lVert S_{\pmb{\lambda}_{T}} e_{v}\rVert^{2} = \sum_{\substack{u\in V_{T}\\ (v,u)\in E_{T}}} \lambda_{u}^{2}.
	\end{equation*} 
	Combining these two equalities we obtain that $ C_{u} = 0 $ for every $ u\in V_{T} $ satisfying $ (v,u)\in E_{T} $.
\end{proof}
As another corollary of the equality \eqref{FormRecursiveNormOfSquare} we obtain that in our setting graphs of isometric composition operators consist only of cycles.
\begin{corollary}
	\label{CorIsometricCompositionOperatorsOneCycle}
	Suppose $ T\!: X\to X $ is such that $ G_{T} $ is connected and satisfies Theorem \ref{ThmClassificationOfGraphs}.(ii). Suppose $ h_{T}\in L^{\infty}(\mu) $. If $ S_{\pmb{\lambda}_{T}}\in \mathbf{B}(\ell^{2}(V_{T})) $ is isometric, then $ \mathcal{R} = \varnothing $, that is, $ V_{T} = \{v_{0},\ldots,v_{\kappa-1} \} $. In particular, $ \ell^{2}(V_{T}) $ is finite dimensional and $ S_{\pmb{\lambda}_{T}} $ is unitary.
\end{corollary}
\begin{proof}
	If $ S_{\pmb{\lambda}_{T}} $ is isometric, then so is $ S_{\pmb{\lambda}_{T}}^{n} $ for every $ n\in\mathbb{N} $. Thus, from \eqref{FormRecursiveNormOfSquare} it follows that
	\begin{equation*}
		\sum_{i = 0}^{\kappa-1}\Lambda_{m,i}^{2}\!\!\!\!\!\sum_{\substack{R\in\mathcal{R}\\i_{R} = i+m\mod\kappa}}\!\!\!\!\! \lambda_{\omega_{R}}^{2}\lVert S_{\pmb{\lambda}_{T}}^{n+\kappa-i-1}e_{\omega_{R}}\rVert^{2} = 0, \qquad m\in \mathbb{N}\cap [0,\kappa-1],\ n\in\mathbb{N}.
	\end{equation*}
	Since $ \Lambda_{m,i} > 0 $ for all $ i,m\in\mathbb{N}\cap[0,\kappa-1] $, we have that
	\begin{equation*}
		\sum_{\substack{R\in\mathcal{R}\\i_{R} = i+m\mod\kappa}}\!\!\!\!\! \lambda_{\omega_{R}}^{2}\lVert S_{\pmb{\lambda}_{T}}^{n+\kappa-i-1}e_{\omega_{R}}\rVert^{2}=0, \qquad i,m\in \mathbb{N}\cap [0,\kappa-1],\ n\in\mathbb{N}.
	\end{equation*}
	From this it follows that
	\begin{equation*}
		\sum_{\substack{R\in\mathcal{R}\\i_{R} = i}} \lambda_{\omega_{R}}^{2} = 0, \qquad i\in\mathbb{N}\cap[0,\kappa-1].
	\end{equation*}
	Since $ \lambda_{\omega_{R}} >0 $ for all $ R\in\mathcal{R} $, the only possibility for the above sum being zero is that the set of indices is empty, that is, $ \mathcal{R} = \varnothing $. This implies that $ \ell^{2}(V_{T}) $ is finite dimensional. Hence, $ S_{\pmb{\lambda}_{T}} $ is invertible and, by \cite[Proposition 1.23]{aglerMisometricTransformationsHilbert1995}, unitary.
\end{proof}
In the end of this section we provide an example of use of Theorem \ref{ThmMIsometricCompositionOperatorsOneCycle}. We characterize 3-isometric composition operators on graphs, in which the only branching points are the vertices on the cycle. Let $ \kappa\in \mathbb{N}_{1} $ and $ \{N_{i}\}_{i=0}^{\kappa-1}\subset \mathbb{N}_{1}\cup\{\infty\} $. Denote
\begin{equation*}
	X = \{0,1,\ldots,\kappa-1\}\cup \{(i,j,k)\!: i\in \mathbb{N}\cap[0,\kappa-1], \ j\in \mathbb{N}\cap[1,N_{i}], \ k\in \mathbb{N} \}.
\end{equation*}
Define $ T\!: X\to X $ by the formula
\begin{equation*}
	T(x) = \begin{cases}
		\kappa-1, \quad \text{if } x = 0\\
		x-1, \quad \text{if } x\in \mathbb{N}\cap[1,\kappa-1]\\
		i, \quad \text{if } x = (i,j,0), 0\le i\le \kappa-1,0\le j\le N_{i}\\
		(i,j,k-1), \quad \text{if } x = (i,j,k), k\ge 1, 0\le i\le \kappa-1,0\le j\le N_{i}
	\end{cases}.
\end{equation*}
Then the graph $ G_{T} = (V_{T},E_{T}) $ takes the form
\begin{align*}
	V_{T} &= X,\\
	E_{T} &= \{(i,(i+1)\mod \kappa)\!: i\in \mathbb{N}\cap[0,\kappa-1]\}\\
	&\cup \{(i,(i,j,0))\!: i\in \mathbb{N}\cap[0,\kappa-1],j\in \mathbb{N}\cap[1,N_{i}]\}\\
	&\cup \{((i,j,k),(i,j,k+1))\!: i\in \mathbb{N}\cap[0,\kappa-1], \ j\in \mathbb{N}\cap[1,N_{i}], \ k\in \mathbb{N}\};
\end{align*}
we present this graph in Figure \ref{ImGraphWithCycleSpec}
\begin{figure}
	\centering
	\includegraphics[scale=0.3]{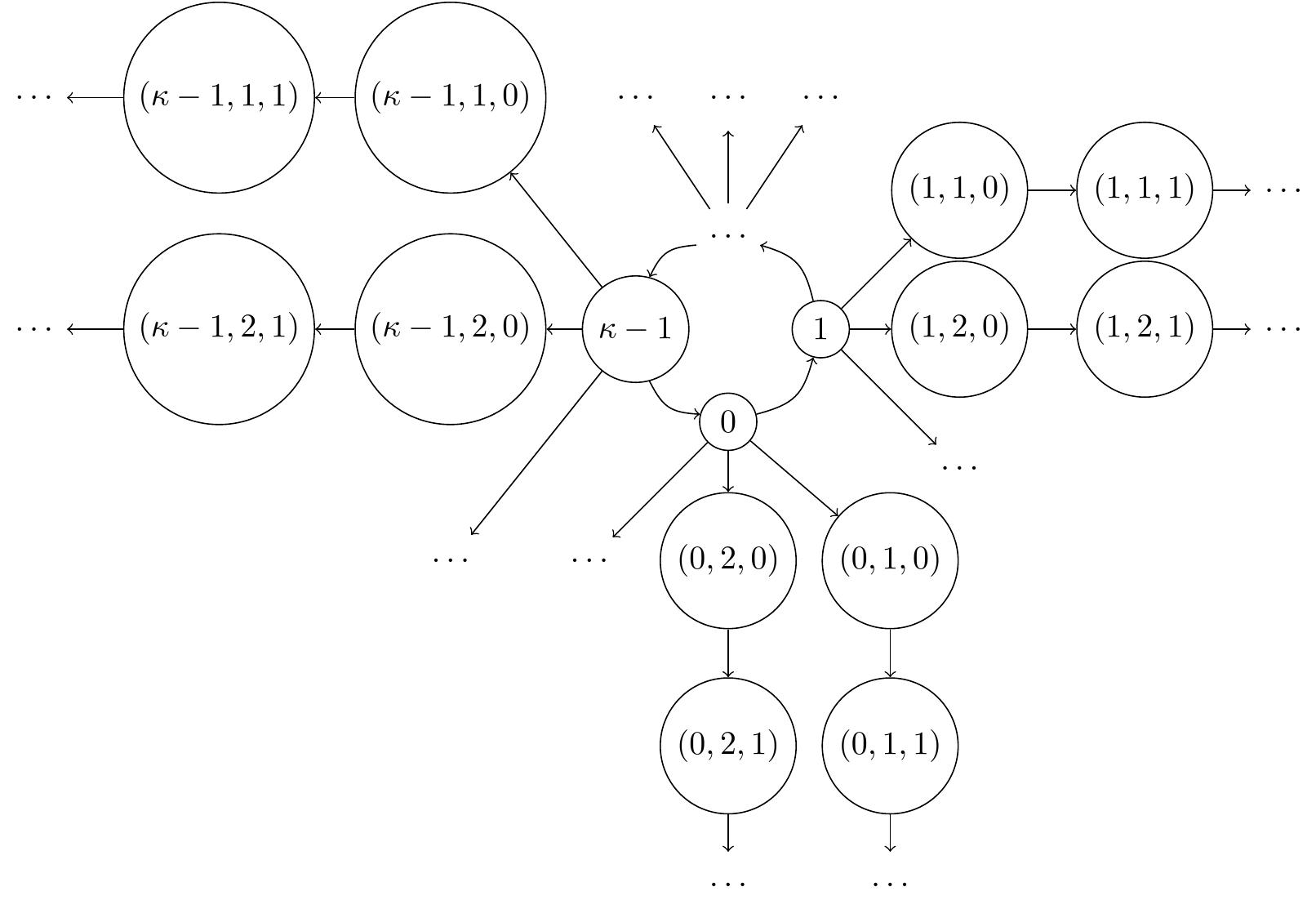}
	\caption{Graph $ G_{T} $, in which the only branching points are the vertices on the cycle}
	\label{ImGraphWithCycleSpec}
\end{figure}
\begin{theorem}
	Let $ X $ and $ T $ be defined as above and let $ \mu\!: 2^{X}\to [0,\infty] $ be a measure on $ X $ such that $ \mu(x) \in (0,\infty) $ for every $ x\in X $. Assume that $ h_{T}\in L^{\infty}(\mu) $. The following conditions are equivalent:
	\begin{enumerate}
		\item $ C_{T}\in \mathbf{B}(L^{2}(\mu)) $ is 3-isometric,
		\item for all $ i\in \mathbb{N}\cap[0,\kappa-1], j\in \mathbb{N}\cap[1,N_{i}] $,
		\begin{equation}
			\label{FormMeasureOnBranches}
			\mu((i,j,k)) = (\mu((i,j,1))-\mu((i,j,0)))k+\mu((i,j,0)), \quad  k\in \mathbb{N}, 
		\end{equation} and\footnote{Note that under the assumption $ h_{T}\in L^{\infty}(\mu) $ we have that for every $ i\in \mathbb{N}\cap[0,\kappa-1] $, $ \sum_{j=1}^{N_{i}}\mu((i,j,0)) = \mu(T^{-1}(\{i\})) \stackrel{\eqref{FormRadonNikodymDerivativeDiscrete}}{=} h_{T}(i)\mu(\{i\}) < \infty $; similar argument shows that $ \sum_{j=1}^{N_{i}}\mu((i,j,1)) < \infty $. By \eqref{FormMeasureOnBranches}, we get that $ \sum_{j=1}^{N_{i}}\mu((i,j,i-\ell-1)) < \infty $ for $ i\in \mathbb{N}\cap[1,\kappa-1] $, $ \ell\in \mathbb{N}\cap[0,i-1] $.}
		\begin{equation}
			\label{FormMeasureOnCycle}
			\mu(i) = q_{0}(i)-\sum_{\ell=0}^{i-1}\sum_{j=1}^{N_{\ell}}\mu((\ell,j,i-\ell-1)), \quad i\in \mathbb{N}\cap[1,\kappa-1],
		\end{equation}
		where $ q_{0}(x) = \frac{A}{2\kappa}x^{2}+\frac{2B-A\kappa}{2\kappa}x+\mu(0) \in \mathbb{R}_{2}[x] $ and
		\begin{align*}
			A &= \sum_{i=0}^{\kappa-1}\sum_{j=1}^{N_{i}}(\mu((i,j,1))-\mu((i,j,0))),\\
			B &= \sum_{i=0}^{\kappa-1}\sum_{j=1}^{N_{i}}\left[(\mu((i,j,1))-\mu((i,j,0)))(\kappa-i-1)+\mu((i,j,0))\right].
		\end{align*}
	\end{enumerate}
\end{theorem}
\begin{proof}
	(i)$ \Longrightarrow $(ii). If $ C_{T} $ is 3-isometric, then so is $ S_{\pmb{\lambda}_{T}} $. By Theorem \ref{ThmMIsometricCompositionOperatorsOneCycle}, this implies that for every $ i\in \mathbb{N}\cap[0,\kappa-1] $ and $ j\in \mathbb{N}\cap[1,N_{i}] $ there exists a polynomial $ p_{i,j}\in \mathbb{R}_{1}[x] $ satisfying \eqref{FormSquareNormPolynomial} with $ v = (i,j,0) $; it is a matter of routine to verify that $ p_{i,j} $ takes the form
	\begin{equation}
		\label{ProofFormPolynomialOnBranches}
		p_{i,j}(x) = (\lambda_{(i,j,1)}^{2}-1)x+1, \qquad x\in \mathbb{R}.
	\end{equation}
	Since, by \eqref{FormDefinitionOfWeights},
	\begin{equation*}
		p_{i,j}(n) = \lVert S_{\pmb{\lambda}_{T}}^{n}e_{(i,j,0)}\rVert^{2} = \prod_{k=1}^{n}\lambda_{i,j,k}^{2} = \frac{\mu((i,j,n))}{\mu((i,j,0))}, \quad n\in \mathbb{N}_{1},
	\end{equation*}
	it follows that the polynomial $ q_{i,j} = \mu((i,j,0))p_{i,j} $ satisfies the equality $ q_{i,j}(k) = \mu((i,j,k)) $ for every $ k\in \mathbb{N} $; moreover, by \eqref{ProofFormPolynomialOnBranches} and \eqref{FormDefinitionOfWeights}, $ q_{i,j} $ takes the form
	\begin{equation*}
		q_{i,j}(x) = (\mu((i,j,1))-\mu((i,j,0)))x+\mu((i,j,0)), \qquad x\in \mathbb{R}.
	\end{equation*}
	Hence, \eqref{FormMeasureOnBranches} holds. Next, by \eqref{FormDefinitionOfWeights} and \eqref{ProofFormPolynomialOnBranches}, the polynomial $ q $ given by \eqref{FormPolynomialP0MinusP0C} takes the form
	\begin{align*}
		q(x) &= \sum_{i=0}^{\kappa-1}\frac{\mu(i)}{\mu(0)} \sum_{j=1}^{N_{i}} \frac{\mu((i,j,0))}{\mu(i)}p_{i,j}(x+\kappa-i-1)\\
		&=\frac{1}{\mu(0)}\sum_{i=0}^{\kappa-1}\sum_{j=1}^{N_{i}}\left[ (\mu((i,j,1))-\mu((i,j,0)))(x+\kappa-i-1)+\mu((i,j,0)) \right]\\
		&=\frac{1}{\mu(0)}\left[ Ax+B \right], \qquad x\in \mathbb{R}.
	\end{align*}
	By Theorem \ref{ThmMIsometricCompositionOperatorsOneCycle}, the polynomial $ p_{0}\in \mathbb{R}_{2}[x] $ satisfying \eqref{FormSquareNormPolynomial} with $ v = 0 $ takes the form $ p_{0}(x) = c_{2}x^{2}+c_{1}x+c_{0} $, where $ (c_{2},c_{1},c_{0}) $ is the solution of the system
	\begin{equation}
		\label{ProofFormCoeffsSystemOfEquations}
		\begin{bmatrix}
			2\kappa & 0 & 0\\
			\kappa^{2} & \kappa & 0\\
			0 & 0 & 1
		\end{bmatrix}\begin{bmatrix}
			c_{2}\\
			c_{1}\\
			c_{0}
		\end{bmatrix} = \begin{bmatrix}
			\frac{1}{\mu(0)}A\\
			\frac{1}{\mu(0)}B\\
			1
		\end{bmatrix}
	\end{equation}
	Solving the above system of equations, we obtain that $ p_{0} = \frac{1}{\mu(0)}q_{0} $. From \eqref{FormNormOfSquareSmallPowers} and \eqref{FormDefinitionOfWeights} it follows that for every $ i\in \mathbb{N}\cap[1,\kappa-1] $,
	\begin{align*}
		q_{0}(i) = \mu(0)\lVert S_{\pmb{\lambda}}^{i}e_{0}\rVert^{2} = \mu(i) + \sum_{\ell=0}^{i-1}\sum_{j=1}^{N_{\ell}}\mu((\ell,j,i-\ell-1)).
	\end{align*}
	After rearrangement we get \eqref{FormMeasureOnCycle}.\\
	(ii)$ \Longrightarrow $(i). Using \eqref{FormMeasureOnBranches}, it can be verified that for every $ i\in \mathbb{N}\cap[0,\kappa-1] $, $ j\in \mathbb{N}\cap[0,N_{i}] $ the polynomial
	\begin{equation*}
		p_{i,j,0}(x) = (\lambda_{(i,j,1)}^{2}-1)x+1, \qquad x\in \mathbb{R},
	\end{equation*}
	satisfies \eqref{FormSquareNormPolynomial} for $ v = (i,j,0) $. Next, from the equality
	\begin{equation*}
		\lVert S_{\pmb{\lambda}_{T}}^{n}e_{(i,j,k)}\rVert^{2} = \frac{\lVert S_{\pmb{\lambda}_{T}}^{n+k}e_{(i,j,0)}\rVert^{2}}{\prod_{\ell=1}^{k}\lambda_{(i,j,\ell)}^{2}}, \qquad k\in \mathbb{N}_{1},
	\end{equation*}
	it follows that for every $ i\in \mathbb{N}\cap[0,\kappa-1] $, $ j\in \mathbb{N}\cap[0,N_{i}] $ and $ k\in \mathbb{N}_{1} $ the polynomial
	\begin{equation*}
		p_{i,j,k}(x) = \left( \prod_{\ell=1}^{k}\lambda_{i,j,\ell}^{2} \right)^{-1}p_{i,j,0}(x+k), \qquad x\in \mathbb{R},
	\end{equation*}
	satisfies \eqref{FormSquareNormPolynomial} with $ v = (i,j,k) $. Assume now $ \kappa > 1 $. We check that $ \rank \tilde{A}_{3,\kappa} = \rank \tilde{B}_{3,\kappa} $. Set $ p_{0} = \frac{1}{\mu(0)}q_{0} $. As in the proof of converse implication, we obtain that the coefficients of $ p_{0} $ satisfies \eqref{ProofFormCoeffsSystemOfEquations}. Combining \eqref{FormMeasureOnCycle} with \eqref{FormNormOfSquareSmallPowers} and \eqref{FormDefinitionOfWeights} we deduce that $ p_{0} $ satisfies
	\begin{equation*}
		\lVert S_{\pmb{\lambda}_{T}}^{n}e_{0}\rVert^{2} = p_{0}(n), \qquad n\in \mathbb{N}\cap[0,\kappa-1].
	\end{equation*}
	This implies that the coefficients of $ p_{0} $ satisfies also the following system of equations:
	\begin{equation*}
		\begin{bmatrix}
			A_{3,\kappa}\\
			B_{3,\kappa}
		\end{bmatrix} \begin{bmatrix}
			c_{2}\\
			c_{1}\\
			c_{0}
		\end{bmatrix} = \begin{bmatrix}
			\frac{1}{\mu(0)}A\\
			\frac{1}{\mu(0)}B\\
			1\\
			a_{1}\\
			\vdots\\
			a_{\kappa-1},
		\end{bmatrix}
	\end{equation*}
	where $ a_{\ell} = \lVert S_{\pmb{\lambda}_{T}}^{\ell}e_{0}\rVert^{2} $, $ \ell\in \mathbb{N}\cap[1,\kappa-1] $. From the Kronecker-Capelli theorem it follows that $ \rank \tilde{A}_{3,\kappa} = \rank \tilde{B}_{3,\kappa} $. The application of Theorem \ref{ThmMIsometricCompositionOperatorsOneCycle} completes the proof.
\end{proof}
	\section{Subnormality of Cauchy dual of composition operator}
\label{SecCauchyDualOfCompOperator}
 In \cite{kosmiderMisometricCompositionOperators2021} the authors studied the Cauchy dual subnormality problem for 2-isometric composition operators with the property that the graph $ G_{T} $ has one cycle and one branching point and obtained the solution in case of graphs with cycle of length 1. In this section we investigate the Cauchy dual subnormality problem for 2-isometric composition operators with graphs having one cycle in full generality. Again, we stick to the notation of Theorem \ref{ThmClassificationOfGraphs}.(ii). First, we are going to derive formulas for $ S_{\pmb{\lambda}_{T}}^{\ast} $ and $ S_{\pmb{\lambda}_{T}}' $.
\begin{lemma}
	\label{LemAdjointModulusCauchyDualFormulas}
	Suppose $ T\!: X\to X $ is such that $ G_{T} $ is connected and satisfies Theorem \ref{ThmClassificationOfGraphs}.(ii). Suppose $ h_{T}\in L^{\infty}(\mu) $. Then
	\begin{enumerate}
		\item $ S_{\pmb{\lambda}_{T}}^{\ast}e_{v} = \lambda_{v}e_{w} $ for every $ v\in V_{T} $, where $ w\in V_{T} $ is the only vertex satisfying $ (w,v)\in E_{T} $, that is $ S_{\pmb{\lambda}_{T}}^{\ast} $ is the weighted shift on $ (V_{T},E_{T}^{-1}) $, where $ E_{T}^{-1} = \{(u,v)\in V_{T}\times V_{T}\!: (v,u)\in E_{T}\} $,
		\item $ S_{\pmb{\lambda}_{T}}^{\ast}S_{\pmb{\lambda}_{T}}e_{v} = \lVert S_{\pmb{\lambda}_{T}}e_{v}\rVert^{2}e_{v} $ for every $ v\in V_{T} $,
		\item if $ S_{\pmb{\lambda}_{T}} $ is left-invertible, then
		\begin{equation*}
			S_{\pmb{\lambda}_{T}}'e_{v} = \sum_{\substack{u\in V_{T}\\ (v,u)\in E_{T}}} c_{v}\lambda_{u}e_{u},
		\end{equation*}
		where $ c_{v} = \lVert S_{\pmb{\lambda}_{T}}e_{v}\rVert^{-2} $, $ v\in V_{T} $; in particular, $ S_{\pmb{\lambda}_{T}}^{\prime} $ is the weighted shift on $ G_{T} $.
	\end{enumerate}
\end{lemma}
\begin{proof}
	(i). Let $ v\in V $. Then
	\begin{equation*}
		\langle S_{\pmb{\lambda}_{T}}e_{w},e_{v}\rangle = \begin{cases}
		\lambda_{v}, & \text{if } (w,v)\in E_{T}\\
		0, & \text{otherwise}
		\end{cases}.
	\end{equation*}
	Hence, $ \langle S_{\pmb{\lambda}_{T}}^{\ast}e_{v},e_{w}\rangle \not= 0 $ if and only if $ (w,v)\in E_{T} $. By Observation \ref{ObsIndegreeOfGraphUniquePathWithDistinctVertices}, there is only one vertex $ w\in V_{T} $ satisfying $ (w,v)\in E_{T} $; for this vertex $ w $ we have $ S_{\pmb{\lambda}_{T}}^{\ast}e_{v} = \lambda_{v}e_{w} $. (ii) and (iii) are simple consequences of (i).
\end{proof}
Note that if $ S_{\pmb{\lambda}_{T}} $ is assumed to be 2-isometric, then, by \cite[Lemma 1]{richterInvariantSubspacesDirichlet1988}, it is also an expansion, so that $ c_{v} \in (0,1] $, $ v\in V_{T} $.\par
Assume that $ G_{T} $ is connected and satisfies Theorem \ref{ThmClassificationOfGraphs}.(ii). Suppose $ S_{\pmb{\lambda}_{T}}\in \mathbf{B}(\ell^{2}(V_{T})) $ is $ 2 $-isometric. For further applications let us emphasize two formulas (cf. Lemma \ref{LemRecursiveFormulas}). For every $ m\in\mathbb{N}\cap[0,\kappa-1] $ and $ n\in\mathbb{N} $ we have
\begin{align*}
	S_{\pmb{\lambda}_{T}}^{\prime n}e_{v_{m}} &= \Lambda_{m,n}^{\prime} e_{v_{m+n\mod \kappa}}\\
	&+\sum_{i = 0}^{n-1}c_{v_{i\mod\kappa}}\Lambda_{m,i}^{\prime} \sum_{\substack{R\in\mathcal{R}\\i_{R} = m+i\mod\kappa}} \lambda_{\omega_{R}} S_{\pmb{\lambda}_{T}}^{\prime n-i-1}e_{\omega_{R}}
\end{align*}
and
\begin{align}
	\label{FormRecursiveCauchyDual}
	S_{\pmb{\lambda}_{T}}^{\prime n+\kappa}e_{v_{m}} &= \Lambda_{m,\kappa}^{\prime} S_{\pmb{\lambda}_{T}}^{\prime n}e_{v_{m}}\nonumber\\
	&+\sum_{i = 0}^{\kappa-1}c_{v_{i\mod\kappa}}\Lambda_{m,i}^{\prime} \sum_{\substack{R\in\mathcal{R}\\i_{R} = m+i\mod\kappa}} \lambda_{\omega_{R}} S_{\pmb{\lambda}_{T}}^{\prime n+\kappa-i-1}e_{\omega_{R}},
\end{align}
where
\begin{equation*}
	\Lambda_{m,i}^{\prime} = \prod_{j=m+1}^{m+i} \lambda_{v_{j\mod\kappa}}c_{v_{(j-1)\mod\kappa}}, \qquad m\in\mathbb{N}\cap[0,\kappa-1],\ i\in\mathbb{N}.
\end{equation*}
The above formulas are consequences of Lemma \ref{LemAdjointModulusCauchyDualFormulas}.(iii) and the fact that for 2-isometric $ S_{\pmb{\lambda}_{T}} $ every $ c_{v} $ ($ v\in V_{T}\setminus\{ v_{0},\ldots,v_{\kappa-1} \} $) is equal 1 (see Lemma \ref{LemOnTreesM-1Isometric}). Note also that $ \Lambda^{\prime}_{m,\kappa} = \Lambda^{\prime}_{0,\kappa} $ for every $ m\in \mathbb{N}\cap[1,\kappa-1] $.\\
It is well-known (see e.g. \cite{anandSolutionCauchyDual2019}) that the Cauchy dual of 2-isometric operator is always a contraction (in our setting it can be easily derived from Lemma \ref{LemAdjointModulusCauchyDualFormulas}.(iii)). This observation will be crucial in our considerations about subnormality of the Cauchy dual of $ S_{\pmb{\lambda}_{T}} $.\par
Note that if $ \mathcal{R}=\varnothing $, then $ \ell^{2}(V_{T}) $ is finite dimensional. In such a case every 2-isometric operator is automatically invertible and, by \cite[Proposition 1.23]{aglerMisometricTransformationsHilbert1995}, unitary. Since $ U' = U $ for every unitary operator $ U $, if $ \mathcal{R} = \varnothing $ and $ S_{\pmb{\lambda}_{T}} $ is 2-isometric, $ S_{\pmb{\lambda}_{T}}^{\prime} $ is obviously subnormal. In the further investigation we always assume that $ \mathcal{R}\not=\varnothing $. Let us present the main result in this section.
\begin{theorem}
	\label{ThmCauchyDualSubnormalityCondition}
	Let $ T\!: X\to X $ be such that $ G_{T} $ is connected, satisfies Theorem \ref{ThmClassificationOfGraphs}.(ii) with additional assumption that $ \mathcal{R}\not=\varnothing $. Assume that $ h_{T}\in L^{\infty}(\mu) $. If $ S_{\pmb{\lambda}}\in \mathbf{B}(\ell^{2}(V_{T})) $ is $ 2 $-isometric, then the following conditions are equivalent:
	\begin{enumerate}
		\item $ S_{\pmb{\lambda}_{T}}' $ is subnormal,
		\item for every $ m\in\mathbb{N}\cap[0,\kappa-1] $,
		\begin{equation*}
			\lVert S_{\pmb{\lambda}_{T}}^{\prime n}e_{v_{m}}\rVert^{2} = \int_{[0,1]} t^{n}\ddd\mu_{v_{m}}(t), \qquad n\in\mathbb{N}\cap[0,\kappa-1],
		\end{equation*}
	\end{enumerate}
	with $ \mu_{v_{m}} = (1-\alpha_{m})\delta_{\sqrt[\kappa]{D}}+\alpha_{m}\delta_{1} $, where
	\begin{align*}
		D &= \Lambda_{0,\kappa}^{\prime 2} =  \prod_{i=0}^{\kappa-1} c_{v_{i}}^{2},\\
		\alpha_{m} &= \frac{C_{m}}{1-D}, \qquad m\in\mathbb{N}\cap[0,\kappa-1],\\
		C_{m} &= \sum_{i=0}^{\kappa-1} c_{v_{i\mod\kappa}}^{2}\Lambda_{m,i}^{\prime 2}\!\!\!\!\!\!\!\sum_{\substack{R\in\mathcal{R}\\i_{R} = i+m\mod\kappa}}\!\!\!\!\!\!\! \lambda_{\omega_{R}}^{2}, \qquad m\in\mathbb{N}\cap[0,\kappa-1].
	\end{align*}
\end{theorem}
\begin{proof}
	(i) $ \Longrightarrow $ (ii). From \cite[Theorems 3.1 and 2.2]{lambertSubnormalityWeightedShifts1976} and the fact that $ S_{\pmb{\lambda}_{T}}^{\prime} $ is contractive it follows that for every $ v\in V_{T} $ there exists the measure $ \mu_{v}\!:\mathcal{B}([0,1])\to [0,1] $ such that
	\begin{equation*}
		\lVert S_{\pmb{\lambda}_{T}}^{\prime n}e_{v}\rVert^{2} = \int_{[0,1]} t^{n}\ddd\mu_{v}(t), \qquad n\in\mathbb{N}.
	\end{equation*}
	In turn, by Theorem \ref{ThmMIsometricCompositionOperatorsOneCycle} and Lemma \ref{LemAdjointModulusCauchyDualFormulas}, we have
	\begin{equation*}
		\lVert S_{\pmb{\lambda}_{T}}^{\prime n}e_{v}\rVert = 1, \quad v\in V_{T}\setminus\{v_{0},\ldots,v_{\kappa-1} \},\ n\in\mathbb{N};
	\end{equation*}
	which implies that $ \mu_{v} = \delta_{1} $ for $ v\in V_{T}\setminus\{v_{0},\ldots,v_{\kappa-1} \} $. From this and \eqref{FormRecursiveCauchyDual} we derive that for every $ m\in\mathbb{N}\cap[0,\kappa-1] $,
	\begin{equation*}
		\lVert S_{\pmb{\lambda}_{T}}^{\prime n+\kappa}e_{v_{m}}\rVert^{2} = D\lVert S_{\pmb{\lambda}_{T}}^{\prime n}e_{v_{m}}\rVert^{2} + C_{m}, \qquad n\in\mathbb{N}.
	\end{equation*}
	Hence,
	\begin{equation*}
		\int_{[0,1]}t^{n+\kappa}\ddd\mu_{v_{m}}(t) = D\int_{[0,1]}t^{n}\ddd\mu_{v_{m}}(t)+C_{m}, \qquad n\in\mathbb{N}.
	\end{equation*}
	Since $ \mu_{v_{m}} $ is regular, it follows that
	\begin{equation}
		\label{ProofFormRecursiveMeasure}
		(t^{\kappa}-D)\ddd\mu_{v_{m}}(t) = C_{m}\delta_{1}, \qquad m\in \mathbb{N}\cap[0,\kappa-1].
	\end{equation}
	Since $ \mathcal{R}\not=\varnothing $, by Corollary \ref{CorIsometricCompositionOperatorsOneCycle} we have that $ D<1 $. This implies that for all $ m\in\mathbb{N}\cap[0,\kappa-1] $, $ \mu_{v_{m}}([0,\sqrt[\kappa]{D})\cup (\sqrt[\kappa]{D},1)) = 0 $.
	Thus,
	\begin{equation}
		\label{ProofFormMeasureOnCycleTwoAtomic}
		\mu_{v_{m}} = (1-\alpha_{m})\delta_{\sqrt[\kappa]{D}}+\alpha_{m}\delta_{1}
	\end{equation} with some $ \alpha_{m} \in [0,1] $. In turn, for every $ m\in\mathbb{N}\cap[0,\kappa-1] $ we have
	\begin{align*}
		D(1-\alpha_{m})+\alpha_{m} &\stackrel{\eqref{ProofFormMeasureOnCycleTwoAtomic}}{=} \int_{[0,1]}t^{\kappa}\ddd\mu_{v_{m}} \\
		&\stackrel{\eqref{ProofFormRecursiveMeasure}}{=} D\mu_{v_{m}}([0,1])+C_{m} = D+C_{m}.
	\end{align*}
	The above equality implies that $ \alpha_{m} = \frac{C_{m}}{1-D} $.\\
	(ii) $ \Longrightarrow $ (i). From Theorem \ref{ThmMIsometricCompositionOperatorsOneCycle} it follows that
	\begin{equation*}
		\lVert S_{\pmb{\lambda}_{T}}^{n}e_{v}\rVert = 1, \quad v\in V_{T}\setminus\{v_{0},\ldots,v_{\kappa-1} \},\ n\in\mathbb{N},
	\end{equation*}
	which implies, by Lemma \ref{LemAdjointModulusCauchyDualFormulas}, that
	\begin{equation*}
		\lVert S_{\pmb{\lambda}_{T}}^{\prime n}e_{v}\rVert = 1, \quad v\in V_{T}\setminus\{v_{0},\ldots,v_{\kappa-1} \},\ n\in\mathbb{N}.
	\end{equation*}
	We check that
	\begin{align}
		\label{FormMomentsOnCycle}
		\lVert S_{\pmb{\lambda}_{T}}^{\prime n}e_{v_{m}}\rVert^{2} = \int_{[0,1]} t^{n}\ddd\mu_{v_{m}}(t), \qquad n\in\mathbb{N}, m\in \mathbb{N}\cap[0,\kappa-1].
	\end{align}
	Since \eqref{FormMomentsOnCycle} holds for $ n\in\mathbb{N}\cap[0,\kappa-1] $, it is enough to check that if \eqref{FormMomentsOnCycle} holds for some $ n\in\mathbb{N} $, then it also holds for $ n+\kappa $. By \eqref{FormRecursiveCauchyDual} we have
	\begin{align*}
		\lVert S_{\pmb{\lambda}_{T}}^{\prime n+\kappa}e_{v_{m}}\rVert^{2} &= D\lVert S_{\pmb{\lambda}_{T}}^{\prime n}e_{v_{m}}\rVert^{2}+C_{m}= D\int_{[0,1]}t^{n}\ddd\mu_{v_{m}}(t) + C_{m}\\
		&= D(1-\alpha_{m})\sqrt[\kappa]{D^{n}}+D\alpha_{m}+C_{m}\\
		&= (1-\alpha_{m})\sqrt[\kappa]{D^{n+\kappa}}+\frac{DC_{m}}{1-D}+C_{m}\\
		&= (1-\alpha_{m})\sqrt[\kappa]{D^{n+\kappa}}+\frac{C_{m}}{1-D}\\
		&= \int_{[0,1]}t^{n+\kappa}\ddd\mu_{v_{m}}(t).
	\end{align*}
	Using \cite[Theorems 3.1 and 2.2]{lambertSubnormalityWeightedShifts1976} we obtain subnormality of $ S_{\pmb{\lambda}_{T}}^{\prime} $.	
\end{proof}
In \cite{kosmiderMisometricCompositionOperators2021} the authors proved that the Cauchy dual of 2-isometric composition operator on graph with one cycle of length 1 and one branching point is subnormal. As a corollary, we get the same result in a more general version.
\begin{corollary}[cf. \mbox{\cite[Theorem 4.6]{kosmiderMisometricCompositionOperators2021}}]
	Assume that $ G_{T} $ is connected and satisfies Theorem \ref{ThmClassificationOfGraphs}.(ii) with additional assumptions that $ \mathcal{R}\not=\varnothing $ and $ \kappa=1 $. If $ S_{\pmb{\lambda}_{T}}\in \mathbf{B}(\ell^{2}(V_{T})) $ is 2-isometric, then $ S_{\pmb{\lambda}_{T}}^{\prime} $ is subnormal.
\end{corollary}
In \cite[Example 6.6]{anandSolutionCauchyDual2019} the authors gave an example of 2-isometry, for which the Cauchy dual is not subnormal. With the help of Theorems \ref{ThmMIsometricCompositionOperatorsOneCycle} and \ref{ThmCauchyDualSubnormalityCondition} we can give another example.
\begin{example}
	Let $ V_{T} = \{0,1,2\}\cup (\{0,1,2\}\times \mathbb{N}) $ and
	\begin{align*}
		E_{T} &= \{(0,1),(1,2),(2,0)\}\cup\{(i,(i,0))\!: i\in\{0,1,2\} \}\\ &\cup \{((i,j),(i,j+1))\!:i\in\{0,1,2\},j\in\mathbb{N}\}.
	\end{align*}
	The graph $ G_{T} = (V_{T},E_{T}) $ is exactly the graph defined at the end of Section \ref{SecMIsometricCompOperators} taken with $ \kappa = 3 $ and $ N_{0} = N_{1} = N_{2} = 1 $ (see Figure \ref{ImGraphWithCycleSpec}). Set $ \pmb{\lambda} = \{\lambda_{0},\lambda_{1},\lambda_{2} \}\cup\{\lambda_{i,j}\!: i\in\{0,1,2 \}, j\in\mathbb{N} \} $ as follows
	\begin{align*}
		\lambda_{1} = a, \ \lambda_{2} = 1, \ \lambda_{0} = \frac{1}{a},\\
		\lambda_{0,0} = \lambda_{1,0} = 1, \ \lambda_{2,0} = z,\\
		\lambda_{i,j} = 1, \qquad i\in\{0,1,2\}, j\in\mathbb{N}_{1},
	\end{align*}
	where $ a,z \in (0,\infty) $. It is easy to verify that $ S_{\pmb{\lambda}}\in\mathbf{B}(\ell^{2}(V_{T})) $; moreover, $ \lVert S_{\pmb{\lambda}}^{n}e_{v}\rVert = 1 $ for every $ v\in (\{0,1,2\})\times \mathbb{N} $ and $ n\in \mathbb{N} $, so the polynomial $ p_{v} = 1 $ satisfies \eqref{FormSquareNormPolynomial} with $ v\in (\{0,1,2\})\times \mathbb{N} $. The polynomial $ q\in \mathbb{R}_{0}[x] $ given by the formula \eqref{FormPolynomialP0MinusP0C} takes the form
	\begin{equation*}
		q(t) = 1+a^{2}+a^{2}z^{2}, \qquad t\in \mathbb{R}.
	\end{equation*}
	Hence, in view of Theorem \ref{ThmMIsometricCompositionOperatorsOneCycle}, $ S_{\pmb{\lambda}} $ is 2-isometric if and only if $ \rank \tilde{A}_{2,3} = \rank \tilde{B}_{2,3} $. The latter holds if and only if the polynomial
	\begin{equation*}
		p_{0}(t) = \frac{1}{3}\left( 1+a^{2}+a^{2}z^{2} \right)t+1, \qquad t\in \mathbb{R},
	\end{equation*}
	satisfies the equalities $ p_{0}(1) = \lVert S_{\pmb{\lambda}}e_{0}\rVert^{2} $ and $ p_{0}(2) = \lVert S_{\pmb{\lambda}}^{2}e_{0}\rVert^{2} $, which is equivalent to the following system of equations
	\begin{align}
		\label{ExampleForm2IsometricPower1}
		a^{2}+1 &= \frac{1}{3}\left( 1+a^{2}+a^{2}z^{2} \right)+1\\
		\label{ExampleForm2IsometricPower2}
		2a^{2}+1 &= \frac{2}{3}\left( 1+a^{2}+a^{2}z^{2} \right)+1
	\end{align}
	Assume $ S_{\pmb{\lambda}} $ is 2-isometric. By Theorem \ref{ThmCauchyDualSubnormalityCondition}, if $ S_{\pmb{\lambda}}^{\prime} $ is subnormal, then there exists $ \alpha\in (0,1) $ satisfying
	\begin{equation*}
		\begin{cases}
			\alpha\sqrt[3]{D}+1-\alpha = \lVert S_{\pmb{\lambda}}^{\prime} e_{0}\rVert^{2}\\
			\alpha\sqrt[3]{D^{2}} + 1-\alpha = \lVert S_{\pmb{\lambda}}^{\prime 2 } e_{0}\rVert^{2}
		\end{cases},
	\end{equation*}
	which implies that
	\begin{equation}
		\label{ExampleFormSystemForCDS}
		\frac{\lVert S_{\pmb{\lambda}}^{\prime} e_{0}\rVert^{2}-1}{\sqrt[3]{D}-1} = \frac{\lVert S_{\pmb{\lambda}}^{\prime 2} e_{0}\rVert^{2}-1}{\sqrt[3]{D^{2}}-1}.
	\end{equation}
	Note that
	\begin{align*}
		\lVert S_{\pmb{\lambda}}^{\prime} e_{0}\rVert^{2} &= \frac{1}{a^{2}+1}\\
		\lVert S_{\pmb{\lambda}}^{\prime 2 } e_{0}\rVert^{2} &= \frac{a^{2}}{2(a^{2}+1)^{2}}+\frac{1}{(a^{2}+1)^{2}}
	\end{align*}
	and
	\begin{equation*}
		D = \frac{1}{4(a^{2}+1)^{2}(a^{-2}+z^{2})^{2}}.
	\end{equation*}
	To obtain 2-isometric operator $ S_{\pmb{\lambda}} $, the Cauchy dual of which is not subnormal, it is enough to find $ a,z\in (0,\infty) $, which satisfy \eqref{ExampleForm2IsometricPower1} and \eqref{ExampleForm2IsometricPower2} and does not satisfy \eqref{ExampleFormSystemForCDS}; this holds if we take, for instance, $ a = \frac{2}{\sqrt{7}} $, $ z = \frac{1}{2} $.
\end{example}

	\section*{Statements and Declarations}
	\textbf{Funding.} The author of the publication received an incentive scholarship from the 
	funds of the program Excellence Initiative - Research University at the Jagiellonian 
	University in Kraków.\par
	\textbf{Competing interests.} The authors have no conflicts of interest to declare that are relevant to the content of this article.
	\section*{Acknowledgements}
	The author would like to thank Zenon Jabłoński for his helpful suggestions concerning the final version of this paper. We are also grateful to anonymous reviewer for the comments improving the exposition quality of the paper.
	\bibliographystyle{plain}
	\bibliography{references}

\end{document}